%% file: main.tex
\documentclass[a4paper,reqno,11pt]{amsart}

\let\oldsubset\subset
\newcommand{\subsetsim}{\mathrel{\ooalign{\raise0.175ex\hbox{$\oldsubset$}\cr\hidewidth\raise-0.9ex\hbox{\scalebox{0.9}{$\sim$}}\hidewidth\cr}}}
\usepackage[T1]{fontenc}
\usepackage[utf8]{inputenc}
\usepackage{amssymb, amsmath, amsthm, graphicx, enumerate, enumitem, color}
\usepackage[usenames,dvipsnames,svgnames,table]{xcolor}
\usepackage[foot]{amsaddr}
\usepackage[normalem]{ulem}

\usepackage[longnamesfirst,numbers,sort&compress]{natbib}

\usepackage{comment}

\makeatletter
\def\thm@space@setup{
\thm@preskip=4mm
\thm@postskip=0mm
}
\makeatother

\usepackage{hyperref}
\hypersetup{
    pdftitle={The Grid-Minor Theorem revisited},
    pdfauthor={Vida Dujmovi\'c, Robert Hickingbotham, Jędrzej Hodor, Gwean\"el Joret, Hoang La, Piotr Micek, Pat Morin, Clément Rambaud, David R.\ Wood},
    colorlinks,
    linkcolor={red!50!black},
    citecolor={blue!50!black},
    urlcolor={blue!80!black}
}

\usepackage{cleveref}

\usepackage{mathtools}

\usepackage{thmtools, thm-restate}

\usepackage{textpos}

\renewenvironment{enumerate}{\begin{enumorig}[label=\textup{(\roman*)}, noitemsep, 
topsep=-7.5pt, labelindent=.2em, leftmargin=*, widest=iii,]}{\end{enumorig}}

\newenvironment{enumerateNum}{\begin{enumorig}
[label=\textup{(\arabic*)}, 
noitemsep, topsep=-7.5pt, labelindent=.2em, leftmargin=*, widest=iii]}{\end{enumorig}}

\newenvironment{enumerateAlpha}{\begin{enumorig}[label=\textup{(\alph*)}, noitemsep, 
topsep=0pt, labelindent=.2em, leftmargin=*, widest=iii]}{\end{enumorig}}

\DeclarePairedDelimiter\set{\{}{\}}

\DeclarePairedDelimiter\floor{\lfloor}{\rfloor}

\theoremstyle{plain}
\newtheorem{thm}{Theorem}
\newtheorem*{thm*}{Theorem}
\newtheorem{theorem}[thm]{Theorem}
\newtheorem{lem}[thm]{Lemma}
\newtheorem{lemma}[thm]{Lemma}
\newtheorem*{lemma*}{Lemma}
\newtheorem{cor}[thm]{Corollary}
\newtheorem*{cor*}{Corollary}
\newtheorem{obs}[thm]{Observation}
\theoremstyle{remark}
\newtheorem{ques}{Question}

\newtheorem*{claim}{Claim}
\crefname{obs}{Observation}{Observations}
\newtheorem*{lem*}{Lemma}
\theoremstyle{definition}

\newtheorem*{conj*}{Conjecture}
\crefname{lem}{Lemma}{Lemmas}
\crefname{thm}{Theorem}{Theorems}
\crefname{cor}{Corollary}{Corollaries}

\newenvironment{proofclaim}[1][]
	{\par\noindent {\it Proof}. }{ \hfill$\lozenge$\par\vspace{11pt}}

\newcommand{\utw}{\operatorname{utw}}
\newcommand{\td}{\operatorname{td}}
\newcommand{\pw}{\operatorname{pw}}
\newcommand{\tw}{\operatorname{tw}}

\newcommand{\cgB}{\mathcal{B}}

\newcommand{\cgF}{\mathcal{F}}
\newcommand{\cgG}{\mathcal{G}} 

\newcommand{\cgI}{\mathcal{I}}

\newcommand{\cgL}{\mathcal{L}}
\newcommand{\cgM}{\mathcal{M}} 
 
\newcommand{\cgN}{\mathcal{N}}

\newcommand{\Oh}{\mathcal{O}}
\newcommand{\cgP}{\mathcal{P}} 
 
\newcommand{\cgR}{\mathcal{R}}
\newcommand{\cgS}{\mathcal{S}}

\newcommand{\bigOh}{\mathcal{O}}

\newcommand{\torso}{\mathrm{torso}}
\newcommand{\dist}{\mathrm{dist}}

\let\le\leqslant

\let\leq\leqslant
\let\geq\geqslant

\let\subset\subseteq

\let\epsilon\varepsilon
\let\epsi\varepsilon
\let\setminus-

\DeclareMathOperator\WReach{WReach}
\DeclareMathOperator\wcol{wcol}

\DeclareMathOperator\len{len}

\DeclareMathOperator\inter{int}

\makeatletter
\let\old@setaddresses\@setaddresses
\def\@setaddresses{\bigskip\bgroup\parindent 0pt\let\scshape\relax\old@setaddresses\egroup}
\makeatother

\newcommand{\piotr}[1]{{\color{orange} Piotr: #1}}

\title{The Grid-Minor Theorem Revisited}

\begin{document}
\author[Dujmovi{\'c}]{Vida Dujmovi{\'c}}
\address[V.~Dujmovi{\'c}]{School of Computer Science and Electrical Engineering, University of Ottawa, Ottawa, Canada}
\email{vida.dujmovic@uottawa.ca}

\author[Hickingbotham]{Robert Hickingbotham}
\email{robert.hickingbotham@monash.edu}

\author[Hodor]{Jędrzej Hodor}
\address[J.~Hodor,\ H.~La,\ P.~Micek]{Theoretical Computer Science Department, 
Faculty of Mathematics and Computer Science, Jagiellonian University, Krak\'ow, Poland}
\email{jedrzej.hodor@gmail.com}

\author[Joret]{Gwena\"el Joret}
\address[G.~Joret]{D\'epartement d'Informatique, Universit\'e libre de Bruxelles, Belgium}
\email{gwenael.joret@ulb.be}

\author[La]{Hoang La}
\email{hoang.la.research@gmail.com}

\author[Micek]{Piotr Micek}
\email{piotr.micek@uj.edu.pl}

\author[Morin]{Pat Morin}
\address[P.~Morin]{School of Computer Science, Carleton University, Ottawa, Canada}
\email{morin@scs.carleton.ca}

\author[Rambaud]{Clément Rambaud}
\address[C.~Rambaud]{DIENS, \'Ecole Normale Sup\'erieure, CNRS, PSL University, Paris, France}
\email{clement.rambaud@ens.psl.eu}

\author[Wood]{David R.\ Wood}
\address[R. Hickingbotham, D.R.~Wood]{School of Mathematics, Monash University, Melbourne, Australia}
\email{david.wood@monash.edu}

\renewcommand{\shortauthors}{Dujmovi{\'c}, Hickingbotham, Hodor, Joret, La, Micek, Morin, Rambaud, Wood} 

\thanks{J.\ Hodor, H.\ La, and P.\ Micek are supported by the National Science Center of Poland under grant UMO-2018/31/G/ST1/03718 within the BEETHOVEN program. 
G.\ Joret is supported by a CDR grant from the Belgian National Fund for Scientific Research (FNRS), by a PDR grant from FNRS, and by the Australian Research Council. Research of R.~Hickingbotham is supported by an  Australian Government Research Training Program Scholarship. Research of V.~Dujmovi\'c is supported by NSERC, and a Gordon Preston Fellowship from the School of Mathematics at Monash University. Research of D.R.~Wood is supported by the Australian Research Council. }

\maketitle

\begin{abstract} 
We prove that for every planar graph $X$ of treedepth $h$, there exists a positive integer $c$ such that 
for every $X$-minor-free graph $G$, 
there exists a graph $H$ of treewidth at most $f(h)$ such that $G$ is isomorphic to a subgraph of $H\boxtimes K_c$.
This is a qualitative strengthening of the Grid-Minor Theorem of Robertson and Seymour~(JCTB, 1986), and treedepth is the optimal parameter in such a result. 
As an example application, 
we use this result to 
improve the upper bound for weak coloring numbers of graphs excluding a fixed graph as a minor. 
\end{abstract}

\newpage


\newpage

\section{Introduction}
\label{sec:introduction}
\input{s.introduction}

\section{Preliminaries}
\label{sec:prelim}
\input{s.preliminaries}

\section{Attached Models}
\label{sec:attached:better}
\input{s.attached_models}

\section{Proof of Theorem~\ref{XMinorFreeProduct}}
\label{sec:main}
\input{s.main_proof}


\section{Proof of Theorem~\ref{thm-wcol-main}}
\label{sec:wcol_main}
\input{s.wcol_main}

\section{Proof of Lemma~\ref{lem:X:geodesics}}
\label{sec:graph-minor-struggles}
\input{s.rest}

\section{Excluding an Apex Graph}
\label{sec:excluding_an_apex}
\input{s.exclude_an_apex}

\section{Open Questions}
\label{sec:open}
\input{s.open_questions}


\bibliographystyle{plain}
\bibliography{biblio}

\end{document}

%% file: s.introduction.tex
The seminal \emph{Graph Minors} series of Robertson and Seymour is the foundation of modern structural graph theory. In this work, treewidth is a central concept that measures how similar a given graph is to a tree. A key theorem of Robertson and Seymour~\citep{RS86V} states  that a minor-closed graph class $\mathcal{G}$ has bounded treewidth if and only if some planar graph is not in $\mathcal{G}$. In particular, for every planar graph $X$, every $X$-minor-free graph has treewidth at most some function $g(X)$. This result is often called the Grid-Minor Theorem since it suffices to prove it when $X$ is a planar grid. The asymptotics of $g$ have been substantially improved since the original work. Most significantly, Chekuri and Chuzhoy~\citep{CC16} showed that $g$ can be chosen to be polynomial in $|V(X)|$. The current best bound is $g(X)\in \widetilde\bigOh(|V(X)|^9)$, which follows from a result of Chuzhoy and Tan~\citep{CT21}. Dependence on $|V(X)|$ is unavoidable, since the complete graph on $|V(X)|-1$ vertices is $X$-minor-free, but has treewidth $|V(X)|-2$. Our goal is to prove a qualitative strengthening of the Grid-Minor Theorem via graph product structure theory.

Graph product structure theory describes graphs in complicated classes as subgraphs of strong products\footnote{The \emph{strong product} $G_1\boxtimes G_2$ of two graphs $G_1$ and $G_2$ is the graph with vertex set $V(G_1\boxtimes G_2):=V(G_1)\times V(G_2)$ and that includes the edge with endpoints $(v,x)$ and $(w,y)$ if and only if
$vw\in E(G_1)$ and $x=y$; $v=w$ and $xy\in E(G_2)$; or
$vw\in E(G_1)$ and $xy\in E(G_2)$.} of simpler graphs. For example, the Planar Graph Product Structure Theorem by Dujmović, Joret, Micek, Morin, Ueckerdt, and Wood~\citep{DJMMUW20} says that for every planar graph $G$ there is a graph $H$ of treewidth at most $3$ and a path $P$ such that $G\subsetsim H \boxtimes P \boxtimes K_3$. Here, $G_1\subsetsim G_2$ means that $G_1$ is isomorphic to a subgraph of $G_2$.

Inspired by this viewpoint, we prove the following product structure extension of the Grid-Minor Theorem. Note that $H\boxtimes K_c$ is the graph obtained from $H$ by `blowing-up' each vertex of $H$ by a complete graph $K_c$.
Let $\tw(G)$ denote the treewidth of a graph $G$, 
and let $\td(G)$ denote the treedepth of $G$ 
(both defined in \Cref{sec:prelim}).

\begin{thm}
\label{thm:selling-point}
For every planar graph $X$, there exists a positive integer $c$ such that  for every $X$-minor-free graph $G$,
there exists a graph $H$ of treewidth at most $2^{\td(X)+1}-4$ such that $G\subsetsim H\boxtimes K_c$. 
\end{thm}

The point of Theorem~\ref{thm:selling-point} is that the treewidth of $H$ only depends on the treedepth of $X$, not on $|V(X)|$
\footnote{Arbitrarily large graphs can have bounded treedepth, such as edgeless graphs (treedepth 1) or stars (treedepth 2).}. The described product structure of $G$ is a more refined description of $G$ compared to the output of the Grid-Minor Theorem
since $\tw(H\boxtimes K_c)\leq (\tw(H)+1)c-1$.
This refinement is useful because various graph parameters can be bounded on $H\boxtimes K_c$ by a fast-growing function of $\tw(H)$ times a slow-growing (usually linear) function of $c$; this includes queue-number~\cite{DJMMUW20}, nonrepetitive chromatic number~\cite{DEJWW20}, 
and others \cite{DFMS21,BDJM23}. 
As concrete applications of \cref{thm:selling-point}, we use it to improve bounds for weak coloring numbers and $p$-centered colorings of $X$-minor-free graphs.

The Grid-Minor Theorem relates to treewidth in the same way as the Excluded-Tree-Minor Theorem by Robertson and Seymour~\cite{GM1} relates to pathwidth.
The latter says that a minor-closed class $\mathcal{G}$ has bounded pathwidth if and only if some tree is not in $\mathcal{G}$. In particular, there is a function $g$ such that for every tree $X$, every $X$-minor-free graph has pathwidth at most $g(|V(X)|)$.
The following product structure version of this result was proved by Dujmović, Hickingbotham, Joret, Micek, Morin, and Wood~\cite{DHJMMW23}: there exists a function $f$ such that 
for every tree $X$, there exists a positive integer $c$ such that for every $X$-minor-free graph $G$,
there exists a graph $H$ of pathwidth at most $f(\td(X))$ such that 
$G\subsetsim H\boxtimes K_c$. 
(In fact, they prove a stronger statement in which the pathwidth of $H$ is bounded by $2h-1$, where $h$ is the radius of $X$.) 

We actually prove the following result for an arbitrary excluded minor, which combined with the Grid-Minor Theorem immediately implies \cref{thm:selling-point}.

\begin{thm}
\label{XMinorFreeProduct}
For every graph $X$, there exists a positive integer $c$ such that for every positive integer $t$ and for every $X$-minor-free graph $G$ with $\tw(G)<t$,  there exists a graph $H$ of treewidth at most $2^{\td(X)+1}-4$ such that 
$G\subsetsim H\boxtimes K_{ct}$.
\end{thm}

\Cref{XMinorFreeProduct} was inspired by and can be restated in terms of the parameter `underlying treewidth' introduced by Campbell, Clinch, Distel, Gollin, Hendrey, Hickingbotham, Huynh, Illingworth, Tamitegama, Tan, and Wood~\cite{CCDGHHHITTW22}. They defined the \emph{underlying treewidth} of a graph class $\mathcal{G}$, denoted by $\utw(\mathcal{G})$, to be 
the smallest integer such that, for some function $f$, for every graph $G\in \cgG$ there is a graph $H$ of treewidth at most $\utw(\mathcal{G})$ such that $G\subsetsim H \boxtimes K_{f(\tw(G))}$. Here, $f$ is called the \emph{treewidth binding function}. 
Campbell~et~al.~\cite{CCDGHHHITTW22} showed that the underlying treewidth of the class of planar graphs equals 3, and the same holds for any fixed surface. More generally, let $\mathcal{G}_X$ be the class of graphs excluding a given graph $X$ as a minor.  Campbell~et~al.~\cite{CCDGHHHITTW22} showed that $\utw(\mathcal{G}_{K_t})=t-2$ and
$\utw(\mathcal{G}_{K_{s,t}})=s$ (for $t\geq\max\{s,3\})$. In these results, the treewidth binding function is quadratic. Illingworth, Scott and Wood~\cite{illingworth2022product} reproved these results with a linear treewidth binding function.  

Determining the underlying treewidth of the class of $X$-minor-free graphs, for an arbitrary graph $X$, was one of the main problems left unsolved by Campbell~et~al.~\cite{CCDGHHHITTW22} and Illingworth, Scott and Wood~\cite{illingworth2022product}. 
\Cref{XMinorFreeProduct} together with a well-known lower bound construction given in~\Cref{LowerBounds} shows that $\utw(\cgG_X)$ and $\td(X)$ are tied:
\begin{equation}\label{eq:bound_on_f}
    \td(X)-2\leq \utw(\cgG_X) \leq 2^{\td(X)+1}-4.
\end{equation}
This shows that treedepth is the right parameter to consider in \cref{thm:selling-point,XMinorFreeProduct}.
Moreover, in the upper bound the treewidth binding function is linear. 

\textbf{Application \#1. Weak Coloring Numbers.}
Weak coloring numbers are a family of graph parameters studied extensively in  structural and algorithmic graph theory. 
See the book by Nešetřil and Ossona de Mendez~\cite{NOM2012}, or the recent lecture notes by Pilipczuk, Pilipczuk, and Siebertz~\cite{notes} for more information on this topic.
For the algorithmic side, see Dvo\v{r}\'{a}k~\cite{Dvorak13} and also Theorem 5.2 in~\cite{notes}, which contains a polynomial-time approximation algorithm for $r$-dominating set, with approximation ratio bounded by a function of a weak coloring number of the input graph.
We now quickly introduce the definition. 
The \emph{length} of a path is the number of its edges.
For two vertices $u$ and $v$ in a graph $G$, a $u$--$v$ \emph{path} is a path in $G$ with endpoints $u$ and $v$. 
Let $G$ be a graph and let $\sigma$ be an ordering of the vertices of $G$. 
For an integer $r\geq0$ and two vertices $u$ and $v$ of $G$, we say that
$u$ is \emph{weakly $r$-reachable} from $v$ in $\sigma$, if there exists
a $u$--$v$ path of length at most $r$ 
such that for every vertex $w$ on the path, $u\leq_{\sigma} w$.
The set of vertices that are weakly $r$-reachable from a vertex $v$ in $\sigma$ is denoted by $\WReach_r[G, \sigma, v]$. 
The $r$-\emph{th weak coloring number} of $G$, denoted by $\wcol_r(G)$, is defined as
\[
\wcol_r(G) = \min_{\sigma}\max_{v \in V(G)}\ |\WReach_r[G, \sigma, v]|.
\]
where $\sigma$ ranges over the set of all vertex orderings of $G$.
Several papers give bounds for weak coloring numbers of graphs in a given sparse class. For example, if $G$ is planar, then $\wcol_r(G)=\Oh(r^3)$; 
if $G$ has no $K_t$-minor, then $\wcol_r(G)=\Oh(r^{t-1})$ as proved by van den Heuvel, Ossona de Mendez, Quiroz, Rabinovich, and Siebertz~\cite{vdHetal17}; 
and if $\tw(G)\leq t$, then 
$\wcol_r(G)\leq \binom{r+t}{t}$ and this bound is tight as proved by Grohe, Kreutzer, Rabinovich,
Siebertz, and Stavropoulos~\cite{Grohe15}. 

Fix a planar graph $X$. What is known about $\wcol_r(G)$ when $X$ is not a minor of $G$?
Since $G$ is $K_{|V(X)|}$-minor-free, we have $\wcol_r(G) = \Oh(r^{|V(X)|-1})$.
However, thanks to~\Cref{thm:selling-point}, 
there exists a graph $H$ with $\tw(H)\leq f(\td(X)) =\Oh(2^{\td(X)})$ and $c$ depending only on $X$ such that 
$G\subsetsim H\boxtimes K_c$. Therefore,
\[
\wcol_r(G) \leq \wcol_r(H\boxtimes K_c) \leq c\cdot\wcol_r(H) \leq c\cdot r^{f(\td(X))}.
\]
Indeed, the first inequality follows from the monotonicity of $\wcol_r$ and the second inequality is an easy property\footnote{
Let $\sigma_H$ be an ordering of $V(H)$ such that $|\WReach_r[H,\sigma,x]| \leq \wcol_r(H)$ for every $x \in V(H)$.
Consider $\sigma$ an ordering of $V(H \boxtimes K_c)$ such that for every $(x,u), (x',u') \in V(H \boxtimes K_x)$ 
if $x <_{\sigma_H} x'$, then $(x,u) <_\sigma (x',u')$.
It is easy to see that $\WReach_r[H\boxtimes K_c,\sigma,(x,u)] \subseteq \{(y,v) \in V(H \boxtimes K_c) \mid y \in \WReach_r[H,\sigma_H,x]\}$, and so $|\WReach_r[H \boxtimes K_c, \sigma, u]| \leq c \cdot \wcol_r(H)$.
} of $\wcol_r$. 
The obtained upper bound on $\wcol_r(G)$ is polynomial in $r$, where the exponent depends only on $\td(X)$ and not on $|V(X)|$.

As mentioned, the Grid-Minor Theorem and also~\Cref{thm:selling-point} hold only 
when the excluded minor $X$ is planar. 
However, \cref{XMinorFreeProduct} does not have this restriction, hence there is no obvious obstacle for the above bound on $\wcol_r(G)$ to hold for all graphs $X$. 
We prove that this is indeed the case, which is the second main contribution of this paper.
\begin{restatable}{theorem}{weakcol}
\label{thm-wcol-main}
There exists a function $g$ such that
for every graph $X$, 
there exists a constant $c$ such that 
for every $X$-minor-free graph $G$ and every positive integer $r$, 
\[
\wcol_r(G) \leq c\cdot r^{g(\td(X))}.
\]
\end{restatable}

Again, the point of~\Cref{thm-wcol-main} is that the degree of the polynomial in $r$ bounding $\wcol_r(G)$ depends only on $\td(X)$ and not on $|V(X)|$. 
In the previous best bound for weak colouring number of $X$-minor free graphs, the degree of the polynomial in $r$ depended on the vertex-cover\footnote{The \emph{vertex-cover number} $\tau(G)$ of a graph $G$ is the size of a smallest set $S \subseteq V(G)$ such that every edge of $G$ has at least one endpoint in $S$.} number $\tau(X)$. 
In particular, it follows from a result by van den Heuvel and Wood~\cite[Proposition 28]{vdHW18} regarding the weak colouring number of $K_{s,t}^\star$-minor-free graphs that $\wcol_r(G) \leq c\cdot r^{\tau(X)+1}$ for every $X$-minor-free graph $G$ and integer $r\geq 1$. \Cref{thm-wcol-main} is qualitatively stronger since $\td(X)\leq \tau(X)+1$ and there are graphs $X$ with $\td(X)=3$ and arbitrarily large $\tau(X)$.

The proof of the theorem relies on the same decomposition lemma as the proof of~\Cref{thm:selling-point}. 
The ordering of the vertices witnessing the bound on $\wcol_r$ in~\Cref{thm-wcol-main} is built via chordal partitions---a powerful proof technique originally developed in the 1980s in the context of the cops and robber game~\cite{A86} that was rediscovered and used in~\cite{vdHetal17} to bound weak coloring numbers, and has subsequently found several other applications in structural graph theory.

\textbf{Application \#2. Product Structure for Apex-Minor-Free Graphs.} As already mentioned, Dujmovi\'c~et~al.~\cite{DJMMUW20} proved that every planar graph is isomorphic to a subgraph of $H\boxtimes P\boxtimes K_3$ for some graph $H$ with treewidth at most 3 and for some path $P$. This result has been the key ingredient in the solution of several open problems on planar graphs~\cite{DJMMUW20,DEJWW20,EJM,DEJGMM21}. Building on this work,  Distel, Hickingbotham, Huynh, and Wood~\cite{DHHW22} proved that every graph of Euler genus $g$ is isomorphic to a subgraph of $H\boxtimes P\boxtimes K_{\max\{2g,3\}}$ for some graph $H$ with treewidth at most 3 and for some path $P$. 
More generally, Dujmovi\'c~et~al.~\cite{DJMMUW20} characterized the graphs $X$ for which there exist integers $t$ and $c$ such that every $X$-minor-free graph is isomorphic to a subgraph of $H\boxtimes P\boxtimes K_c$ where $\tw(H)\leq t$ and $P$ is a path. 
The answer is precisely the apex graphs. 
Here a graph $X$ is \emph{apex} if $V(X)=\emptyset$ or $X-u$ is planar for some vertex $u$ of $X$. 
The following natural problem arises: for a given apex graph $X$, what is the minimum integer $t(X)$ such that, for some integer $c$, every $X$-minor-free graph is isomorphic to a subgraph of $H\boxtimes P\boxtimes K_c$ where $\tw(H)\leq t(X)$ and $P$ is a path? Illingworth, Scott and Wood~\cite{illingworth2022product} showed that $t(X)\leq \tau(X)$. 
We show, via an application of \cref{XMinorFreeProduct}, that  $t$ is tied to treedepth. In particular, 
\begin{equation}
    \label{ApexStructure}
    \td(X)-2 \leq t(X) \leq 2^{\td(X)+1}-1.
\end{equation}
The proof of this result is presented in \Cref{sec:excluding_an_apex}. 

\textbf{\boldmath Application \#3. $p$-Centered Colorings.}
\Cref{thm:selling-point} can also be used to improve bounds for $p$-centered chromatic numbers of graphs excluding a fixed minor.
For an integer $p\geq 1$, a vertex coloring $\phi$ of a graph $G$ is \emph{$p$-centered} if for every connected subgraph $H$ of $G$ either $\phi$ uses more than $p$ colors in $H$ or there is a color that appears exactly once in $H$. The \emph{$p$-centered chromatic number} of $G$, denoted by $\chi_p(G)$, is the least number of colors in a $p$-centered coloring of $G$.
Centered colourings are important since they characterize graph classes of bounded expansion~\cite{NOM2012}, and are a central tool for designing parameterized algorithms in classes of bounded expansion~\cite{PS19,PW18}; see \cite{DFMS21} for an overview. 

If $\tw(G)\leq t$ then $\chi_p(G)\leq \binom{p+t}{t}$, and this bound is again tight~\cite{PS19,DFMS21}. 
If $X$ is a planar graph and $G$ is $X$-minor-free, then by~\Cref{thm:selling-point},
there exists a graph $H$ with $\tw(H)\leq f(\td(X))$ and $c$ depending only on $X$ such that 
$G\subsetsim H\boxtimes K_c$. Therefore,
\[
\chi_p(G) \leq \chi_p(H\boxtimes K_c) \leq c\cdot\chi_p(H) \leq c\cdot p^{f(\td(X))},
\]
where the first inequality follows from the monotonicity of $\chi_p$ and the second inequality is an easy property of $p$-centered colorings (Lemma~8 in~\cite{DFMS21}). The obtained upper bound on $\chi_p(G)$ is polynomial in $p$, where the exponent depends only on $\td(X)$. Similarly, for an apex graph $X$, we use
\eqref{ApexStructure} to show that for some $c=c(X)$, every $X$-minor-free graph $G$ satisfies $\chi_p(G) \leq c\cdot p^{f(\td(X))}$. 

\textbf{Outline.} 
\Cref{sec:prelim} gives all the necessary definitions, as well as some preliminary results about tree-decompositions. 
\Cref{LowerBounds} proves the lower bounds in \eqref{eq:bound_on_f} and \eqref{ApexStructure}. 
\Cref{sec:attached:better} provides a decomposition lemma (\Cref{lemma:find_rooted_cliques}) for graphs avoiding an `attached model' of a fixed graph, which is a key ingredient in the results that follow. 
This part of the argument, in particular \Cref{lemma:find_attached_minors_or_cuts_preserve_model}, is inspired by a result of Kawarabayashi~\cite{kawarabayashi_rooted_2004} on rooted minors, which in turn is inspired by results of Robertson and Seymour~\cite{robertson_graph_1995}. 
\Cref{sec:main} contains the proof of~\Cref{XMinorFreeProduct}.
\Cref{sec:wcol_main} contains the proof of~\Cref{thm-wcol-main}, which relies on chordal partitions (see Lemma~\ref{lemma:Heuvel_et_al_wcol_Kt_minor_free} and Figure~\ref{fig:wcol_Kt_minor_free}), 
and a variant of the Helly property for $K_t$-minor-free graphs that is of independent interest (see~\Cref{lem:X:geodesics}). 
Section~\ref{sec:graph-minor-struggles} contains the proof of~\Cref{lem:X:geodesics}. 
This part builds on the work of Pilipczuk and Siebertz~\cite{PS19} for bounded genus graphs, 
and on the Graph Minor Structure Theorem of Robertson and Seymour~\cite{GMXVI} for $K_t$-minor-free graphs. 
\Cref{sec:excluding_an_apex} presents a product structure decomposition for graphs excluding an apex graph of small treedepth as a minor. As a consequence, we obtain better bounds for the $p$-centered chromatic number of such graphs. \Cref{sec:open} concludes with four questions that we find relevant and exciting for future work.

%% file: s.preliminaries.tex
For a positive integer $k$, we use the notation $[k]=\{1,\ldots,k\}$, 
and when $k=0$ let $[k] = \emptyset$. The \emph{empty graph} is the graph with no vertices. 
All graphs considered in this paper are finite and may be empty. 

Let $G$ be a graph.
Recall that the \emph{length} of a path $P$, denoted by $\len(P)$, is the number of edges of $P$.
The \emph{distance} between two vertices $u$ and $v$ in $G$, denoted by $\dist_G(u,v)$, is the minimal length of a path with endpoints $u,v$ in $G$, 
if such a path exists, and $+\infty$ otherwise.
A path $P$ is a \emph{geodesic} in $G$ if it is a shortest path between its endpoints in $G$.

Let $u$ be a vertex of $G$.
The \emph{neighborhood} of $u$ in $G$, denoted by $N_G(u)$, is the set $\{v \in V(G) \mid uv \in E(G)\}$.
For every set $X$ of vertices of $G$, let $N_G(X)=\bigcup_{u \in X} N_G(u) \setminus X$.
For every integer $r \geq 1$, we denote by $N_G^r[u] = \{v \in V(G) \mid \dist_G(u,v) \leq r\}$.
We omit $G$ in the subscripts when it is clear from the context.

A \emph{rooted forest} is a disjoint union of rooted trees.
The \emph{vertex-height} of a rooted forest $F$ is the maximum number of vertices on a path from a root to a leaf in $F$.  
For two vertices $u$, $v$ in a rooted forest $F$, we say that $u$ is a \emph{descendant} of $v$ in $F$
if $v$ lies on the path from a root to $u$ in $F$.  
The \emph{closure} of $F$ 
is the graph with vertex set $V(F)$ and edge set $\set{vw\mid v \text{ is a strict descendant of $w$ in $F$}}$. 
The \emph{treedepth} of a graph $G$, denoted by $\td(G)$, is $0$ if $G$ is empty, and otherwise is the minimum vertex-height of a rooted forest $F$ with $V(F)=V(G)$ such that $G$ is a subgraph of the closure of $F$.

Consider the following family of graphs $\set{U_{h,d}}$. 
For every positive integer $d$, define $U_{0,d}$ to be the empty graph. 
For all positive integers $h$ and $d$, define 
$U_{h,d}$ to be the closure of the disjoint union of $d$ complete $d$-ary trees of vertex-height $h$. 
Observe that $U_{h,d}$ has treedepth $h$. 
Moreover, this family of graphs is universal for graphs of bounded treedepth: For every graph $X$ of treedepth at most $h$, there exists $d$ such that $X \subsetsim U_{h,d}$. Thus, every $X$-minor-free graph is $U_{h,d}$-minor-free, and to prove \Cref{XMinorFreeProduct} it suffices to do so for $X=U_{h,d}$.  

A \emph{tree-decomposition} $\mathcal{W}$ of a graph $G$ is a pair $(T,(W_x \mid x\in V(T)))$, 
where $T$ is a tree and the sets $W_x$ for each $x \in V(T)$ are subsets of $V(G)$ 
called \emph{bags} satisfying:
\begin{enumerate}
\item for each edge $uv\in E(G)$ there is a bag containing both $u$ and $v$, and
\item for each vertex $v\in V(G)$ the set of vertices $x\in V(T)$ with 
$v\in W_x$  induces a non-empty subtree of $T$.
\end{enumerate}
The {\em width} of $\mathcal{W}$ is $\max\{|W_x|-1 \mid x\in V(T) \}$, and its \emph{adhesion} is $\max\{|W_x\cap W_y| \mid xy \in E(T)\}$. 
The \emph{treewidth} of a graph $G$, denoted $\tw(G)$, is the minimum width of a tree-decomposition of $G$.


A \emph{clique} in a graph is a set of pairwise adjacent vertices. Given two graphs $G_1,G_2$, a clique $K^1$ in $G_1$, a clique $K^2$ in $G_2$, a function $f:K^2\to K^1$, the \emph{clique-sum} of $G_1$ and $G_2$ according to $f$ is the graph $G$ obtained from the disjoint union of $G_1$ and $G_2$ by identifying $x$ with $f(x)$ for every $x \in K^2$. Note that $f$ does not have to be injective. It is well known that $\tw(G) \leq \max\{\tw(G_1),\tw(G_2)\}$. 

Given two graphs $G$ and $H$, an \emph{$H$-partition} of $G$ is a partition $(V_x \mid x \in V(H))$ of $V(G)$ with possibly empty parts
such that for all distinct $x,y\in V(H)$, 
if there is an edge between $V_x$ and $V_y$ in $G$, then $xy \in E(H)$. 
The \emph{width} of such an $H$-partition is $\max\{|V_x| \colon x \in V(H)\}$.

\begin{obs}[Observation 35 in~\cite{DJMMUW20}]
\label{ObsPartitionProduct}
For all graphs $G$ and $H$, and every positive integer $c$, $G\subsetsim H\boxtimes K_c$ if and only if $G$ has an $H$-partition of width at most $c$.
\end{obs}

A \emph{partition} of a graph $G$ is a family $\mathcal{P}$ of induced subgraphs of $G$ such that every vertex in $G$ is in the vertex set of exactly one member of $\mathcal{P}$.
Given a partition $\mathcal{P}$ of $G$, define $G/\mathcal{P}$
to be the graph with vertex set $\mathcal{P}$ and edge set all the pairs $P,Q \in \mathcal{P}$ such that there is an edge between $V(P)$ and $V(Q)$ in $G$. 

A \emph{layering} of a graph $G$ is a sequence $(L_0,L_1,\ldots)$ of disjoint subsets of $V(G)$ whose union is $V(G)$ and such that for every edge $uv$ of $G$ there is a non-negative integer $i$ such that $u,v \in L_i \cup L_{i+1}$.
A \emph{layered partition} of $G$ is a pair $(\mathcal{P},\mathcal{L})$ where $\mathcal{P}$ is a vertex partition of $G$ and $\mathcal{L}$ is a layering of $G$.

\begin{obs}[Observation 35 in~\cite{DJMMUW20}]
\label{ObsLayeredPartitionProduct}
For all graphs $G$ and $H$, and every positive integer $c$, $G\subsetsim H\boxtimes P  \boxtimes K_c$ for some path $P$ if and only if
there is an $H$-partition $(V_x \mid x \in V(H))$ of $G$ and a layering $\mathcal{L}$ such that $|V_x \cap L| \leq c$ for every $x \in V(H)$ and $L \in \mathcal{L}$.
\end{obs}



We finish these preliminaries with three simple statements on tree-decompositions.

\begin{lemma}[Statement (8.7) in~\cite{RS86V}]\label{lemma:helly_property_tree_decomposition}
    For every graph $G$, for every tree-decomposition $\mathcal{W}$ of $G$, for every family $\mathcal{F}$ of connected subgraphs of $G$, for every positive integer $d$, either:
    \begin{enumerate}
        \item there are $d$ pairwise vertex-disjoint subgraphs in $\mathcal{F}$, or
        \item there is a set $S$ that is the union of at most $d-1$ bags of $\mathcal{W}$ such that $V(F) \cap S \neq \emptyset$ for every $F \in \mathcal{F}$.
    \end{enumerate}
\end{lemma}

A tree-decomposition
$(T,(W_x\mid x\in V(T)))$ of a graph $G$ is said to be \emph{natural} if for every edge $e$ in $T$, 
for each component $T_0$ of $T-e$, the graph $G\left[\bigcup_{z\in V(T_0)} W_z\right]\ \textrm{is connected.}$ 
The following statement appeared first in~\cite{FN06}, see also~\cite{GJNW23}.
\begin{lemma}[Theorem~1 in~\cite{FN06}]\label{lemma:natural_tree_decomposition}
    Let $G$ be a connected graph and let $(T,(W_x\mid x\in V(T)))$ be a tree-decomposition of $G$. 
    There exists a natural tree-decomposition $(T',(W'_x\mid x\in V(T')))$ of $G$ such that 
    for every $x'\in V(T')$ there is $x\in V(T)$ with 
    $W'_{x'}\subseteq W_x$.
\end{lemma}

The following technical lemma encapsulates a step in the main proof.  
In this lemma, we ``capture'' a given set of vertices $Y$ with a superset $X$ such that $X$ is not too large and each component of $G-X$ has a bounded number of neighbors in $X$.

\begin{lemma}\label{lemma:increase_X_to_have_small_interfaces}
Let $m$ be a positive integer.
Let $G$ be a graph and let $\mathcal{W}$ be a tree-decomposition of $G$.
If $Y$ is the union of $m$ bags of $\mathcal{W}$, then there is a set $X$ that is the union of at most $2m-1$ bags of $\mathcal{W}$ such that $Y\subseteq X$ and for every component $C$ of $G-X$, $N(V(C)) \cap X$ is a subset of the union of at most two bags of $\mathcal{W}$.
Moreover, if $\mathcal{W}$ is natural, then $N(V(C)) \cap X$ intersects at most two components of $G-V(C)$.
\end{lemma}

\begin{proof}
    First, we prove the following claim for rooted trees.
    \begin{claim}
        Let $T$ be a rooted tree, let $r$ be the root of $T$, and let $U$ be a non-empty subset of $V(T)$. Then, there exists $V \subseteq V(T)$ such that $U \subseteq V$, $|V| \leq 2|U|-1$, and for each component $C$ of $T-V$:
        \begin{enumerate}
            \item\label{item:root_component} if $r \in V(C)$, then $C$ is adjacent to at most one vertex of $V$;
            \item\label{item:non_root_component} otherwise, $C$ is adjacent to at most two vertices of $V$.
        \end{enumerate}
    \end{claim}

    \begin{proofclaim}
        We prove the claim by induction on $|V(T)|$.
        For a 1-vertex tree $T$, the statement holds with $V=U$.
        Now assume that $|V(T)|>1$.
        Let $T_1, \dots, T_{d}$ be the rooted subtrees of $T-r$, where $d$ is the degree of $r$ in $T$. 
        If $T_i$ is disjoint from $U$ for some $i \in [d]$, then apply induction to $T-T_i$ and $U$. The obtained set $V$ satisfies the claim also for $T$.
        Therefore, without loss of generality, we assume that for every $i \in[d]$, $T_i$ intersects $U$, and let $U_i = V(T_i) \cap U$.

        First, suppose that $d=1$. 
        By induction applied to $T_1$ and $U_1$, we obtain $V_1$ satisfying the assertion of the claim. 
        Let $V = V_1$ if $r \notin U$, and $V = V_1 \cup \{r\}$ otherwise.
        One can immediately verify that $V$ satisfies the assertion for $T$.

        Next, suppose that $d > 1$.
        For each $i \in [d]$, by induction applied to $T_i$ and $U_i$, we obtain $V_i$ satisfying the assertion.
        Let $V := \{r\} \cup \bigcup_{i \in [d]} V_i$.
        Clearly, $U \subset V$.
        Consider a component $C$ of $T-V$. Then $C$ is a component of $T_i-V$ for some $i\in [d]$. 
        Since $r \in V$, we have $r \notin C$, and so,~\ref{item:root_component} holds.
        If $C$ is not adjacent to $r$, then~\ref{item:non_root_component} is satisfied by induction.
        If $C$ is adjacent to $r$, then  the root of $T_i$ is in $C$, thus, by induction, $C$ is adjacent to at most one vertex in $V_i$, and so, it is adjacent to at most two vertices in $V$.
        Finally, 
        $|V|\leq 1+\sum_{i\in [d]}|V_i|
        \leq 1+\sum_{i\in [d]}(2|U_i|-1) 
        = (2\sum_{i \in [d]}|U_i|) - (d-1) 
        \leq 2|U| - (d-1) \leq 2|U| - 1$. 
    \end{proofclaim}

    Now, we prove the lemma.
    Let $\mathcal{W} = (T,(W_x \mid x \in V(T)))$.
    Let $U$ be a set of $m$ vertices in $T$ such that $Y \subseteq \bigcup_{x \in U} W_u$.
    By the claim, there exists $V \subseteq V(T)$ of size at most $2m-1$ such that $U \subseteq V$ and every component of $T-V$ has at most two neighbors in $V$.
    Let $X := \bigcup_{x \in V} W_x$. 
    For a given component $C$ of $G-X$, let $T_C$ be the subtree of $T$ induced by the vertices $x \in V(T)$ with $W_x \cap V(C) \neq \emptyset$.
    Observe that $T_C$ is connected and disjoint from $V$, and so $S = N_T(V(T_C)) \cap V$ has size at most two.
    Finally, $N(V(C)) \cap X \subseteq \bigcup_{s \in S} W_s$.
    Moreover, if $\mathcal{W}$ is natural, then for every $s \in S$, $W_s$ is in a single component of $G-V(C)$.
\end{proof}

\section{Improper Colourings and Lower Bounds}
\label{LowerBounds}

This section explores connections between our results and  improper graph colorings, which lead to the lower bound on $\utw(\cgG_X)$ in \eqref{eq:bound_on_f} and the lower bound on $t(X)$ in \eqref{ApexStructure}. 

A graph $G$ is \emph{$k$-colorable with defect $d$} if each vertex can be assigned one of $k$ colors such that each monochromatic subgraph has maximum degree at most $d$. A graph $G$ is \emph{$k$-colorable with clustering $c$} if each vertex can be assigned one of $k$ colors such that each monochromatic connected subgraph has at most $c$ vertices. The \emph{defective chromatic number} of a graph class $\mathcal{G}$ is the minimum integer $k$ such that for some integer $d$, every graph in $\mathcal{G}$ is $k$-colorable with defect $d$.  Similarly, the \emph{clustered chromatic number} of a graph class $\mathcal{G}$ is the minimum integer $k$ such that for some integer $c$, every graph in $\mathcal{G}$ is $k$-colorable with clustering $c$. These topics have been widely studied in recent years; see \citep{DEMWW22,vdHW18,OOW19,NSSW19,LW1,LW2,LW3,Liu22,WoodSurvey,EJ14,LO18} for example. Clustered coloring is closely related to the results in this paper, since a graph $G$ is $k$-colorable with clustering $c$ if and only if $G\subsetsim H\boxtimes K_c$ for some graph $H$ with $\chi(H)\leq k$. Our results are stronger in that they replace the condition $\chi(H)\leq k$ by the qualitatively stronger statement that $\tw(H)\leq k$ (since $\chi(H)\leq \tw(H)+1$). Of course, this is only possible when $G$ itself has bounded treewidth. 

The treedepth of $X$ is the right parameter to consider when studying the defective or clustered chromatic number of the class of $X$-minor-free graphs. Fix any connected\footnote{These results hold for possibly disconnected $X$, but with treedepth replaced by a variant parameter called connected treedepth, which differs from treedepth by at most 1.} graph $X$ with treedepth $h$. Ossona de Mendez, Oum and Wood~\cite[Proposition~6.6.]{OOW19} proved that the defective chromatic number of the class of $X$-minor-free graphs is at least $h-1$, and conjectured that equality holds. Norin, Scott, Seymour and Wood~\citep{NSSW19} proved a relaxation of this conjecture with an exponential bound, and in the stronger setting of clustered coloring. In particular, they showed that every $X$-minor-free graph is $(2^{h+1}-4)$-colorable with clustering $c(X)$. 
The proof of Norin~et~al.~\citep{NSSW19} went via treewidth. In particular, they showed that every $X$-minor-free graph with treewidth $t$ is $(2^{h}-2)$-colorable with clustering $ct$ where $c=c(X)$; that is, $G\subsetsim H\boxtimes K_{ct}$ for some graph $H$ with $\chi(H)\leq 2^h-2$.
\Cref{XMinorFreeProduct} provides a qualitative strengthening of this result by showing that $G\subsetsim H\boxtimes K_{ct}$ for some graph $H$ with $\tw(H)\leq 2^{h+1}$ where $c=c(X)$. Liu~\cite{Liu22} recently established the original conjecture of Ossona de Mendez~et~al.~\citep{OOW19}, which also implies that the clustered chromatic number of $X$-minor-free graphs is at most $3h-3$, by a result of~Liu and Oum~\cite[Theorem~1.5]{LO18}. 

For the sake of completeness, we now adapt the argument of Ossona de Mendez~et~al.~\cite{OOW19} to conclude the lower bound in \eqref{eq:bound_on_f}  on underlying treewidth, and the lower bound in \eqref{ApexStructure} related to the product structure of apex-minor-free graphs. We start with the following well-known statement (see \citep[Lemma~12]{CCDGHHHITTW22} for a similar result).




\begin{lemma}
\label{lem-uhd}
Let $h,d$ be positive integers, and let $H$ be a graph.
For every $H$-partition of $U_{h,d}$ of width at most $d$, we have $\tw(H)\geq h-1$.
\end{lemma}
\begin{proof}
Let $(V_x \mid x \in V(H))$ be an $H$-partition of $U_{h,d}$ with width at most $d$. 
Recall that $U_{h,d}$ is the closure of the disjoint union of $d$ complete $d$-ary trees of vertex-height $h$. In what follows, we refer to these underlying complete $d$-ary trees when we consider parent/child relations, subtrees  rooted at a given vertex, and leaves. For every $x \in V(H)$, every vertex $u \in V_x$ that is not a leaf in $U_{h,d}$ has a child $v$ such that the subtree rooted at $v$ in $U_{h,d}$ is disjoint from $V_x$. This implies that there is a sequence $u_1, \dots, u_{h}$ of vertices in $U_{h,d}$ such that $u_{i+1}$ is a child of $u_i$ for every $i \in [h-1]$, and $u_i \in V_{x_i}$ for every $i \in [h]$ with $x_1, \dots, x_{h}$ pairwise distinct. Since $\{u_1, \dots, u_{h}\}$ is a clique in $U_{h,d}$, $\{x_1, \dots, x_{h}\}$ is a clique in $H$. This shows that $K_{h} \subsetsim H$, which implies $\tw(H)\geq  h-1$.
\end{proof}

The next lemma proves the lower bound in \eqref{eq:bound_on_f}.

\begin{lemma}
\label{TreeDepthLowerBound}
For every graph $X$, 
$\utw(\cgG_X)\geq\td(X)-2$. 
\end{lemma}
\begin{proof}
    Let $X$ be a graph and let $h=\td(X)-1$. By the definition of $\utw(\cdot)$ together with \Cref{ObsPartitionProduct}, there exists an integer-valued function $f$ such that every $X$-minor-free graph $G$ has an $H$-partition of width at most $f(\tw(G))$ for some graph $H$ of treewidth at most $\utw(\cgG_X)$. Let $d=f(\td(X)-2)$. 
    Note that $X$ has larger treedepth than $U_{h,d}$, therefore $U_{h,d}\in \cgG_X$.
    By~\Cref{lem-uhd}, every $H$-partition of $U_{h,d}$ of width at most $d$ satisfies $\tw(H)\geq h-1$. 
    Hence $\utw(\cgG_X)\geq h-1 = \td(X)-2$.
\end{proof}



The next result, which is an adaptation of Theorem~19 in \citep{DJMMUW20}, proves the lower bound in \eqref{ApexStructure}.


\begin{lem}\label{LowerBoundtX}
    Let $c$ be a positive integer, and let $X$ be a graph.
    There exists an $X$-minor-free graph $G$ such that for every graph $H$ and every path $P$, if $G\subsetsim H\boxtimes P\boxtimes K_c$, then $\tw(H)\geq \td(X)-2$.
\end{lem}

\begin{proof}
    Fix $h=\td(X)-1$ and $d=3c$. Since $h=\td(U_{h,d})>\td(X)$, we conclude that $U_{h,d}$ is $X$-minor-free. 
    Now suppose that $U_{h,d} \subsetsim H\boxtimes P\boxtimes K_c$ for some graph $H$ and path $P$. 
    We claim that $\tw(H)\geq h-1=\td(X)-2$, 
    which would complete the proof.

By \Cref{ObsLayeredPartitionProduct} there is an $H$-partition $(V_x \mid x \in V(H))$ of $U_{h,d}$ and a layering $\mathcal{L}$ such that $|V_x \cap L| \leq c$ for every $x \in V(H)$ and $L \in \mathcal{L}$. Since $U_{h,d}$ has radius $1$, any layering of $U_{h,d}$ has at most three layers. So $|V_x|\leq 3c$ for every $x\in V(H)$. Thus $(V_x \mid x \in V(H))$ is an $H$-partition of $U_{h,d}$ with width at most $3c$.
Now \Cref{lem-uhd} implies $\tw(H)\geq h-1=\td(X)-2$, as desired.
\end{proof}

%% file: s.attached_models.tex
Let $G$ and $H$ be graphs. Then $H$ is a \emph{minor} of $G$ if a graph isomorphic to $H$ can be obtained from $G$ by deleting edges, deleting vertices and contracting edges. If $H$ is not a minor of $G$, then $G$ is \emph{$H$-minor-free}. A \emph{model} of $H$ in $G$ is a family $(B_x \mid x \in V(H))$ of pairwise disjoint subsets of $V(G)$ such that: 
\begin{enumerate}
    \item for every $x \in V(H)$, the subgraph induced by $B_x$ is non-empty and connected.
    \item for every edge $xy \in E(H)$, there is an edge between $B_x$ and $B_y$ in $G$.
\end{enumerate}
The sets $B_x$ for $x \in V(H)$ are called the \emph{branch sets} of the model. Note that $H$ is a minor of $G$ if and only if there is an $H$-model in $G$.

The \emph{join} of graphs $G_1$ and $G_2$, denoted by $G_1 \oplus G_2$, is the graph obtained from the disjoint union of $G_1$ and $G_2$ by adding all edges between vertices in $G_1$ and vertices in $G_2$. Similarly, given a set $U$ and a graph $G$ with $V(G)\cap U = \emptyset$,  denote by $U\oplus G$ the graph with vertex set $U \cup V(G)$ and edge set $E(G) \cup\{uv \mid u \in U, v \in V(G)\} \cup \{uu' \mid u,u' \in U, u\neq u'\}$.

Let $G$ and $H$ be graphs. Let $a$ and $k$ be integers with $a\geq k\geq0$, and let $R_1,\ldots,R_ k$ be pairwise disjoint subsets of $V(G)$. 
A model $(B_v\mid v\in V(K_a\oplus H))$ of $K_a\oplus H$ in $G-\bigcup_{i=1}^k R_i$ is \emph{$\{R_1,\dots,R_k\}$-attached} in $G$
if there are $k$ distinct vertices $v_1,\ldots,v_k$ of $K_a$ such that 
$B_{v_i}$ contains a neighbor of a vertex in $R_i$ in $G$ for each $i\in[k]$.
If $R = \{r_1, \dots, r_k\} \subseteq V(G)$ is a set of $k$ vertices, then
we say that a model of $K_a\oplus H$ in $G$ is \emph{$R$-attached} in $G$ if it is $\{\{r_1\},\dots, \{r_k\}\}$-attached in $G$.

In this paper, a \emph{separation} in $G$ is a pair $(A,B)$ of subgraphs of $G$ such that $A \cup B=G$ (where $V(A) \subseteq V(B)$ or $V(B) \subseteq V(A)$ is allowed). The \emph{order} of $(A,B)$ is $|V(A) \cap V(B)|$.
Let $G$ be a graph and $S,T$ be two sets of vertices of $G$.
Let $k$ be a positive integer.
A \emph{linkage} of order $k$ between $S$ and $T$ is a family of $k$ vertex-disjoint paths from $S$ to $T$ in $G$, with no internal vertices in  $S\cup T$. 
Menger's Theorem asserts that either $G$ contains a linkage of order $k$ between $S$ and $T$ or 
there is a separation $(A,B)$ of $G$ of order at most $k-1$ such that $S \subseteq V(A)$ and $T \subseteq V(B)$.

The next lemma and  corollary are tools to be used in the main decomposition lemma that follows (\Cref{lemma:find_all_cuts_and_the_simple_part_in_the_middle}). 
\Cref{lemma:find_attached_minors_or_cuts_preserve_model} is inspired by a result of Kawarabayashi~\cite{kawarabayashi_rooted_2004}. 

\newcommand{\aaa}{a}

For a graph $G$ and a subset $R$ of the vertices of $G$, let $G^{+R}$ be the graph obtained from $G$ by adding all missing edges between vertices of $R$.

\begin{lemma}
\label{lemma:find_attached_minors_or_cuts_preserve_model}
Let $H$ be a graph, and let $\aaa$ and $k$ be positive integers with $\aaa \geq 2k$.
For every graph $G$ and every set $R$ of $k$ vertices of $G$ such that there exists a model $\mathcal{M}=(B_x \mid x\in V(K_\aaa\oplus H))$ of $K_\aaa \oplus H$ in $G^{+R}$, at least one of the following 
properties hold:
\begin{enumerate}
    \item
    \label{item:claim_attached_model}
    $G$ contains an $R$-attached model of $K_{\aaa-k} \oplus H'$, for some graph $H'$ obtained from $H$ by removing at most $2k$ vertices, 
    \item
    \label{item:claim_R_cut} 
    there is a separation $(A,B)$ in $G$ of order at most $k-1$ and a vertex $z$ in $K_a$ 
    such that 
    $R\subseteq V(A)$ and $B_z\subseteq V(B)-V(A)$.
\end{enumerate}
\end{lemma}

\begin{proof}
    Suppose that the lemma is false and let $G$ be a graph with the minimum number of vertices for which there exist $R$ and $\cgM$ as in the statement such that neither~\ref{item:claim_attached_model}~nor~\ref{item:claim_R_cut} holds. Fix such a set $R$ and model $\cgM = (B_x \mid x \in V(K_a\oplus H))$.

    We claim that for each $x\in V(K_a\oplus H)$ we have $B_x\subseteq R$ or $B_x$ is a singleton. 
    Suppose the opposite, that is, 
    there exists a branch set $U$ of $ \cgM$ such that $|U| > 1$ and $U$ is not a subset of $R$. In particular, $G[U]$ contains an edge $e=uv$ such that $u \in V(G) - R$ and $v \in V(G)$.
    Consider the graph $G_1$ obtained from $G$ by contracting $e$. Contracting an edge inside a branch set of a model preserves the model. 
    Let $\mathcal{M}_1 = (B^1_x \mid x \in V(K_a\oplus H))$ be the resulting model of $K_\aaa\oplus H$ in $G_1^{+R}$. By the minimality of $G$, the lemma holds for $G_1$, $R$, and $\cgM_1$. 
    If item~\ref{item:claim_attached_model} holds, that is, 
    $G_1$ contains an $R$-attached model of $K_{\aaa-k} \oplus H'$, where $H'$ is a graph obtained from $H$ by removing at most $2k$ vertices, then $G$ does as well, a contradiction.
    Therefore, item~\ref{item:claim_R_cut} holds and we fix a separation $(A_1,B_1)$ in $G_1$ of order at most $k-1$ and $z\in V(K_a)$ such that 
    $R \subseteq V(A_1)$ and $B^1_z\subseteq V(B_1)-V(A_1)$.
    By uncontracting $e$, we obtain a separation $(A,B)$ in $G$ of order at most $k$ such that 
    $R \subseteq V(A)$ and $B_z\subseteq V(B)-V(A)$. 
    This separation has to be of order exactly $k$, in particular, $u$ and $v$ are both in $V(A) \cap V(B)$, as otherwise, \cref{item:claim_R_cut} would be satisfied for $G$, $R$, and $\cgM$.
    
    Let $R' = V(A) \cap V(B)$.
    By Menger's Theorem, either there exists a linkage of order $k$ between $R$ and $R'$ in $A$, or there exists a separation $(C,D)$ in $A$ of order at most $k-1$ such that $R\subseteq V(C)$ and $R'\subseteq V(D)$. 
    In the latter case, we obtain a separation $(C,D\cup B)$ in $G$ of order at most $k-1$ such that 
    $R\subset V(C)$ and $B_z\subseteq V(D\cup B)-V(C)$. Thus,~\ref{item:claim_R_cut} is satisfied for $G$,$R$, and $\cgM$, which is a contradiction. 
    Therefore, there exists a linkage $\mathcal{L}$ of order $k$ between $R$ and $R'$ in $A$.
    Since $|R|=k$, $|R'|=k$, and not all vertices of $R'$ are in $R$ (since $u\in R'-R$), at least one vertex of $R$ is in $V(A) - V(B)$. 
    Since $z$ is adjacent to every other vertex in $K_a\oplus H$ and $B_z \subseteq V(B)-V(A)$, 
    every branch set $Y$ in $\mathcal{M}$ contains a vertex of $B$, and thus  $B^{+R'}[Y]$ is non-empty and connected. 
    Let $\mathcal{M}'=(B'_x \mid x\in V(K_a\oplus H))$ be obtained from $\mathcal{M}$ by restricting each branch set to the graph $B$. 
    It follows that $\mathcal{M}'$ is a model of $K_a\oplus H$ in $B^{+R'}$. 
    Since $B$ has fewer vertices than $G$, the triple $B,R',\mathcal{M}'$ satisfy the lemma. 
    If \cref{item:claim_attached_model} is satisfied, that is, if there is an $R'$-attached model of $K_{\aaa-k} \oplus H'$ in $B$, where $H'$ is a graph obtained from $H$ by removing at most $2k$ vertices, then we can extend the model using $\mathcal{L}$ to obtain an $R$-attached model of $K_{a-k} \oplus H'$ in $G$, a contradiction.
    Therefore, \cref{item:claim_R_cut} is satisfied for $B,R',\mathcal{M}'$, that is, there is a separation $(A',B')$ in $B$ of order at most $k-1$ 
    and $z'\in V(K_a\oplus H)$ with $R' \subseteq V(A')$ and $B_{z'}\subseteq V(B')-V(A')$. Observe that $(A \cup A',B')$ is a separation in $G$ of order at most $k-1$ such that $R\subset V(A \cup A')$ and $B_{z'}\subseteq V(B') - V(A\cup A')$, a contradiction.
    This proves that each branch set of $\cgM$ is either a singleton or a subset of $R$.

    Let $M$ be the union of all branch sets in $\mathcal{M}$ that does not intersect $R$.
    By Menger's Theorem, either there is a linkage of order $k$ between $R$ and $M$ in $G$, or there is a separation $(A,B)$ of $G$ of order at most $k-1$ with $R \subseteq V(A)$, $M \subseteq V(B)$. 
    Suppose the latter is true. 
    Observe that for every vertex $z$ in $K_a$, the corresponding branch set $B_z$ is either contained in $M$ or intersects $V(A)\cap V(B)$, since $z$ is adjacent to every other vertex of $K_a\oplus H$ (and $M$ is not empty). 
    Thus, since $a > k-1$, there is a choice of $z$ such that $B_z$ is disjoint from $V(A)$.     
    Hence, item~\ref{item:claim_R_cut} holds.
    Now, assume that there is a linkage $\mathcal{L}$ of order $k$ between $R$ and $M$ in $G$.
    
    Let $W_{1K},W_{1H},W_{2K},W_{2H},W_{3K},W_{3H}$ be the partition of $V(K_a\oplus H)$ defined by $V(K_a)=\bigcup_{i\in[3]}W_{iK}$, $V(H)=\bigcup_{i\in[3]}W_{iH}$, and
    \begin{align*}
    W_{1K}\cup W_{1H} &= \{x \in V(K_a\oplus H) \mid B_x \subseteq R\},\\
    W_{2K}\cup W_{2H} &= \{x \in V(K_a\oplus H) \mid B_x \subseteq \bigcup_{L \in \mathcal{L}} V(L) - R\},\\
    W_{3K}\cup W_{3H} &= V(K_a\oplus H) - (W_{1K}\cup W_{1H} \cup W_{2K}\cup W_{2H}).
    \end{align*}
    See \Cref{fig:lemma:using:linkage:sacrifice} for an illustration. First, we argue that $|W_{2H}| \leq |W_{3K}|$. Observe that
        \begin{align*}
            a &= |W_{1K}| + |W_{2K}| + |W_{3K}|,\\
            k &= |W_{2H}| + |W_{2K}|.
        \end{align*}
    Combining the above with $a \geq 2k$ and $|W_{1K}| \leq k$, we obtain
        \[|W_{3K}| = a-|W_{1K}| - |W_{2K}| = a-|W_{1K}|-(k - |W_{2H}|) \geq 2k-k-k+|W_{2H}| = |W_{2H}|.\]
    It follows that there exists an injective mapping $f: W_{2H} \rightarrow W_{3K}$.

    Let $\aaa'= \aaa - |W_{1K}|$, and let $H' = H - (W_{1,H} \cup W_{2H})$. Note that $H'$ is a graph obtained from $H$ by removing at most $2k$ vertices.
    Now, define a model $\mathcal{M}'=(B'_x\mid x\in W_{2K}\cup W_{3K}\cup V(H'))$ 
    of $K_{\aaa'} \oplus H'$ as follows:
    \[
    B'_x=
    \begin{cases}
    B_x \cup B_{f^{-1}(x)}&\textrm{ if $x \in f(W_{2H}) \subset W_{3K}$,} \\
    B_{x}&\textrm{ otherwise.} \\
    \end{cases}
    \]
    See \Cref{fig:lemma:using:linkage:sacrifice} again. Now, $\mathcal{L}$ is a linkage of order $k$ between $R$ and $k$ distinct branch sets $B'_x$ with $x \in V(W_{2K} \cup W_{3K})$. We can extend the model using $\cgL$. Namely, for each path in $\cgL$ we add all its internal vertices to the unique branch set that intersects the path. We obtain an $R$-attached model of $K_{\aaa'} \oplus H'$, hence, \ref{item:claim_attached_model} holds. This contradiction concludes the proof.
    \begin{figure}[!htbp]
       \centering 
       \includegraphics[scale=1]{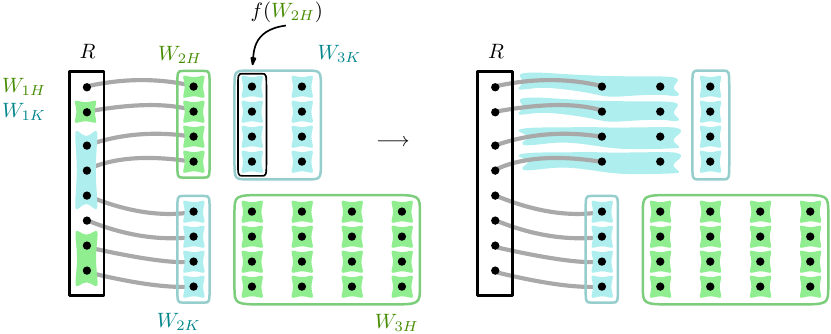} 
       \caption{Illustration of the proof of \Cref{lemma:find_attached_minors_or_cuts_preserve_model}.
       Edges are not drawn. 
       On the left, we show an example of a model $\cgM$ of $K_a \oplus H$. Each branch set is either a singleton or is contained in $R$. 
       The blue shapes are the branch sets of the vertices in $K_a$. 
       The green shapes are the branch sets of the vertices of $H$. 
       Bold lines represent the linkage. 
       We mark all the sets $W_{iC}$ for $i \in \{1,2,3\}$ and $C \in \{K,H\}$. 
       Now we briefly recall the process of obtaining $\cgM'$ from $\cgM$ described in the proof of \Cref{lemma:find_attached_minors_or_cuts_preserve_model}.
       The result of the process is depicted on the right of the figure.
       First, remove all the branch sets contained in $R$, that is, the branch sets corresponding to vertices of $W_{1K}$ and $W_{1H}$. 
       In an $R$-attached model, we have to attach blue branch sets to $R$. 
       Therefore, we enlarge the blue branch sets in $f(W_{2H})$ by merging them with the ones in $W_{2H}$. 
       We lost at most $k$ blue branch sets and at most $2k$ green branch sets. 
       Thus, the new model is a model of $K_{a'}\oplus H'$, where $a' \geq a-k$, and $H'$ is a graph obtained from $H$ by removing at most $2k$ vertices. 
       Finally, use the linkage to extend the branch sets so that the model is $R$-attached.} 
       \label{fig:lemma:using:linkage:sacrifice} 
    \end{figure} 
\end{proof}

\begin{cor}\label{lemma:find_rooted_cliques}
    Let $G$ be a connected graph, let $k$ be a positive integer, and let $R$ be a set of $k$ vertices of $G$.
    If $K_{2k}$ is a minor of $G^{+R}$, then for some $\ell\in[k]$ there is a separation $(A,B)$ in $G$ of order $\ell$ such that $R\subseteq V(A)$, and $B$ contains a $V(A) \cap V(B)$-attached model of~$K_{\ell}$.
\end{cor}


\begin{proof}
    We proceed by induction on $k$. For $k=1$, one can take $A$ to be a 1-vertex graph containing the vertex of $R$, and $B = G$. 
    Note that $B - A$ is non empty, and since $G$ is connected, there is a vertex in $B-A$ adjacent to the vertex in $R$. This vertex constitutes a $V(A) \cap V(B)$-attached model of $K_1$.
    Now, assume that $k\geq 2$ and that the result holds for all positive integers less than $k$. 
    Let $\cgM=(B_x \mid x\in V(K_{2k}))$ be a model of $K_{2k}$ in $G^{+R}$. 
    Apply \Cref{lemma:find_attached_minors_or_cuts_preserve_model}  to $G,R,\cgM$ with $H$ being the empty graph and $a = 2k$. If item~\ref{item:claim_attached_model} is satisfied, then take $A$ to be the graph on $R$ with no edges and $B$ to be the whole graph $G$, and the lemma is satisfied with $\ell=k$. 
    Otherwise, there exists a separation $(C,D)$ in $G$ of order at most $k-1$ 
    and $z\in V(K_a)$
    such that $R \subseteq V(C)$ and $B_z \subseteq V(D)-V(C)$.
    Let $E$ be the component of $D$ containing $B_z$. 
    Since $z$ is adjacent to every other vertex in $K_{2k}$, 
    $B_x$ contains a vertex of $E$ for every $x\in V(K_{2k})$. 
    Let $\cgM_E$ be obtained from $\cgM$ by replacing each branch set in $\cgM$ by its restriction to $E$. 
    Let $R' = V(C) \cap V(E)$. Thus, $|R'|\leq k-1$. 
    Observe that $\cgM_E$ is a model of $K_{2k}$ in $E^{+R'}$.
    By induction applied to $E$ and $R'$, there exists a separation $(A',B')$ of order at most $k-1$ in $E$ such that 
    $R'\subseteq V(A')$ and $B'$ has a $V(A') \cap V(B')$-attached model of $K_{|V(A')\cap V(B')|}$.     
    Finally, put $A = C \cup A'$ and $B = B'\cup (D-E)$. 
\end{proof}

A crucial step in the proof of \Cref{XMinorFreeProduct} relies on the following lemma, which decomposes graphs 
that do not have some attached models. 

\begin{lemma}\label{lemma:find_all_cuts_and_the_simple_part_in_the_middle}
    Let $G$ be a graph, let $h,a,k,d$ be integers with 
    $h,d\geq1$ and $a \geq k\geq0$, 
    and let $R$ be a set of $k$ vertices of $G$.
    If $G$ contains no $R$-attached model of $K_a \oplus U_{h,d}$, then there is an induced subgraph $C$ of $G$ 
    such that $R \subseteq V(C)$ and the following items hold.
    \begin{enumerate}
        \item
        \label{item:lemma:find_all_cuts_ii} 
        Let $m$ be the number of components of $G-C$, let $C^1,\ldots C^m$ be these components, and let $N^i=N_G(V(C^i))$ for every $i\in [m]$. 
        For every $i\in[m]$, $|N^i|\leq k-1$ and
        $G[V(C^i)\cup N^i]$ has an $N^i$-attached model of $K_{|N^i|}$. 
        \item
        \label{item:lemma:find_all_cuts_i}
        Let $C^0$ be the graph obtained from $C-R$ by adding all missing edges between vertices of $N^i$ for every $i\in[m]$ ($C^0$ is a minor of $G-R$ by \ref{item:lemma:find_all_cuts_ii}).
        Then, $C^0$ is $(K_{a+k}\oplus U_{h,d+2k})$-minor-free. 
    \end{enumerate}
\end{lemma}

\begin{figure}[!htbp]
   \centering 
   \includegraphics[scale=0.95]{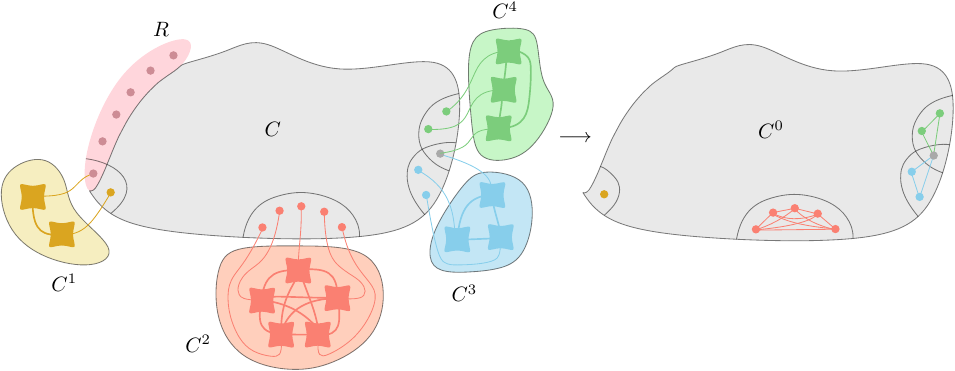}  
   \caption{Illustration of \Cref{lemma:find_all_cuts_and_the_simple_part_in_the_middle} with $a=k=6$. We assume that the initial graph has no $R$-attached model of $K_{6} \oplus U_{h,d}$. We have $m=4$, $|N^1| = 2$, $|N^2| = 5$, $|N^3| = 3$, and $|N^4|=3$. By contracting the connected components $C^i$ in the right way and removing $R$ we obtain the graph $C^0$, which is $(K_{6 + 6} \oplus U_{h,d+2\cdot 6})$-minor-free.}
   \label{fig:components:C} 
 \end{figure}

\begin{proof}
See \Cref{fig:components:C} for an illustration of the assertion. We proceed by induction on $|V(G)|$. Clearly, if $G-R$ is $(K_{a+k}\oplus U_{h,d+2k})$-minor-free, then $C=G$ is the required graph. In particular, this is always the case when $k=0$.

Now, assume that $G-R$ contains a model
$\mathcal{M}=(B_x\mid y\in V(K_{a+k} \oplus U_{h,d+2k}))$ of $K_{a+k}\oplus U_{h,d+2k}$ and $k \geq 1$. 
Apply~\Cref{lemma:find_attached_minors_or_cuts_preserve_model} with $H = U_{h,d+2k}$.  
Observe that every graph obtained from $U_{h,d+2k}$ by removing at most $2k$ vertices contains $U_{h,d}$ as an induced subgraph.
Therefore, since $G$ does not contain an $R$-attached model of $K_a\oplus U_{h,d}$,  \cref{item:claim_attached_model} in \Cref{lemma:find_attached_minors_or_cuts_preserve_model} does not hold. 
It follows that there exists a separation $(A,B)$ of order at most $k-1$ and 
$z\in V(K_{a+k})$
such that $R \subseteq V(A)$ and $B_z\subseteq V(B)-V(A)$.
We can assume that $B-A$ is connected. 
Indeed, if $B-A$ is disconnected, then let $D$ be a component of $B-A$ containing $B_z$ and replace $(A,B)$ with the separation $(G[V(A)\cup V(B)-V(D)],G[V(D)\cup (V(A)\cap V(B))])$. 

If $(A,B)$ is a separation of order $0$, then 
by induction applied to $G-B$, there exists an induced subgraph $C$ of $G-B$ such that $R \subset V(C)$ and items~\ref{item:lemma:find_all_cuts_ii}-\ref{item:lemma:find_all_cuts_i} hold. Components of $G-C$ are components of $(G-B)-C$ and $B$. 
Hence, $C$ also witnesses the assertion of the lemma for $G$.
Therefore, we assume that $(A,B)$ is of order at least $1$.

Let $R'=V(A)\cap V(B)$. Note that $R'$ is non-empty and $|R'|\leq k-1$.
Since $z$ is adjacent to the remaining $a+k-1$ vertices of $K_{a+k}$ and $B_z\subseteq V(B)-V(A)$, 
for every $x \in V(K_{a+k})$ the set $B_x$ contains a vertex of $B$.
Let $\cgM_B$ be obtained from $\cgM$ by replacing each branch set in $\cgM$ by its restriction to $V(B)$. 
Observe that $\cgM_B$ is a model of $K_{a+k}$ in $B^{+R'}$.
By~\Cref{lemma:find_rooted_cliques}
applied to $B$ and $R'$, there is a separation $(E,F)$ in $B$ such that if $R'' = V(E) \cap V(F)$, then $1 \leq |R''| \leq |R'| \leq k-1$,  and 
$F$ contains an $R''$-attached model of $K_{|R''|}$.
Like before, we can assume that $F-E$ is connected.
Let $G'$ be the graph obtained 
from $A \cup E$ by adding all missing edges between vertices in $R''$. 
The model of $K_{|R''|}$ in $F-E$ is disjoint from $E$. 
Thus, $G'$ has fewer vertices than $G$. 
Hence, by induction, $G'$ contains an induced subgraph $C'$ such that 
$R \subset V(C')$ and items~\ref{item:lemma:find_all_cuts_ii}-\ref{item:lemma:find_all_cuts_i} hold.
Let $C^1,\dots,C^m$ be the connected components of $G'-C'$, and let $N^i = N_{G'}(V(C^i))$ for every $i \in [m]$.
We claim that $C = G[V(C')]$ satisfies items~\ref{item:lemma:find_all_cuts_ii}-\ref{item:lemma:find_all_cuts_i}.  
Note that $C$ and $C'$ have the same set of vertices. 
Since $G'[R'']$ is a complete graph, either $R'' \subset V(C)$, or $R'' \subset V(C^i) \cup N^i$ for some $i \in [m]$. 
In the first case, $C^1,\dots,C^m$, and $C^{m+1}=F-E$ are the components of $G-C$. Observe that $N^{m+1}=R''$ and items~\ref{item:lemma:find_all_cuts_ii}-\ref{item:lemma:find_all_cuts_i} hold.
In the second case, $R'' \subset V(C^i) \cup N^i$ for some $i \in [m]$ and $R''$ is not a subset of $V(C)$. In this case, $C^i \cup (F-E)$ is a connected component of $G-C$, and so, both items of the assertion follow immediately.
\end{proof}

%% file: s.main_proof.tex
This section proves 
\Cref{XMinorFreeProduct}. As argued in \Cref{sec:prelim}, it suffices to do so for $X=U_{h,d}$. 

Let $\tau : \mathbb{Z}^2_{\geq 0} \rightarrow \mathbb{Z}$ be the function defined by
\begin{align*}
    \tau(0,k) &= k-2, \textrm{ and}\\ 
    \tau(h,k) &= \tau(h-1,2k+1)+k+1\ \textrm{for every $h\geq 1$,}
\end{align*}
for every $k \geq 0$.
One can check that 
$\tau(h,0) = 2^{h+1}-4$ for every $h \geq 0$.

Moreover, let $c: \mathbb{Z}_{\geq 0}^3 \rightarrow \mathbb{Z}$ be the function defined by 
\begin{align*}
    c(0,d,k)&=1,\ \text{ and}\\
    c(h,d,k)&= \max\{d-1,\ 2,\ k,\ c(h-1,d+2k,2k+1), \,
     2(d-1)2^k-1\} \text{ for every $h\geq1$,}
\end{align*}
for every $d,k \geq 0$.

A key to our proof of \Cref{XMinorFreeProduct} is to prove the following stronger result for  $K_{k} \oplus U_{h,d}$-minor-free graphs. 

\begin{lemma}\label{UhdMinorFreePartition}
For all integers $h,d,t\geq1$ and $k\geq0$, for every $K_{k} \oplus U_{h,d}$-minor-free graph $G$ with  $\tw(G) < t$, there exists a graph $H$ with treewidth at most $\tau(h,k)$, and an $H$-partition $(V_x \mid x \in V(H) )$ of $G$ such that $|V_x|\leq c(h,d,k) \cdot t$ for all $x\in V(H)$.
\end{lemma}

This result with $k=0$ and 
\Cref{ObsPartitionProduct} implies \Cref{XMinorFreeProduct}. The proof of \Cref{UhdMinorFreePartition} is by induction on $h$. Considering  
$K_{k} \oplus U_{h,d}$-minor-free graphs enables the proof to trade-off a decrease in $h$ with an increase in $k$. 

We will need the following result by Illingworth, Scott, and Wood~\cite{illingworth2022product} for the base case of our induction.

\begin{theorem}[Theorem 4 in~\cite{illingworth2022product}]
\label{theorem:illingworth_scott_wood_K_k_minor_free}
    For all integers $k\geq2$ and $t\geq 1$, for every $K_k$-minor-free graph $G$ with $\tw(G) < t$, there is a graph $H$ of treewidth at most $k-2$ and an $H$-partition of $G$ of width at most $t$.
\end{theorem}

The next lemma with $\ell=0$ immediately implies \Cref{UhdMinorFreePartition}, which in turn implies \Cref{thm:selling-point} and \Cref{XMinorFreeProduct}. 
    
\begin{lemma}\label{lemma:main_induction}
    For all integers $h$, $d$, $k$, $\ell$, $t$ with
    $h,d,t\geq1$, $k\geq0$, and $0 \leq \ell \leq k$, 
    for every graph $G$ such that $K_{k} \oplus U_{h,d}$ is not a minor of $G$, $\tw(G) < t$, and for all pairwise disjoint non-empty subsets $R_1,\dots,R_{\ell}$ of vertices of $G$ such that $|R_j| \leq 2$ for every $j \in [\ell]$, there exists a graph $H$, an $H$-partition $(V_x \mid x \in V(H) )$ of $G$, and $x_1,\dots,x_{\ell}  \in V(H)$ such that:
    \begin{enumerateNum}
        \item $\tw(H) \leq \tau(h,k)$, \label{item:pfirst}\label{item:twtau}
        \item $|V_x|\leq c(h,d,k) \cdot t$ for all $x\in V(H)$,\label{item:pwx}
        \item $R_j = V_{x_j}$ for all $j \in [\ell]$,\label{item:RV}
        \item $\{x_1,\dots,x_{\ell}\}$ is a clique in $H$.\label{item:Hcomplete}
        \label{item:plast}
    \end{enumerateNum}
\end{lemma}

\begin{proof}
    We call a tuple 
    $(h,d,k,t,G,\set{R_1,\ldots,R_\ell})$ satisfying the premise of the lemma an \emph{instance}. We proceed by induction on $(h,|V(G)|)$ in lexicographic order.

    If $h=1$ and $k=0$, then $K_k \oplus U_{h,d}$ is the graph with $d$ vertices and no edges.
    Thus, $|V(G)| \leq d-1$ and $\{V(G)\}$ is a $K_1$-partition of $G$ of width at most $d-1$.
    Then, items \ref{item:RV} and \ref{item:Hcomplete} hold vacuously, and items \ref{item:twtau} and \ref{item:pwx} are clear since $\tw(K_1) = 0 = \tau(1,0)$ and $d-1 \leq c(1,d,0)$. From now on, assume that $(h,k) \neq (1,0)$.

    If $|V(G)-\bigcup_{j \in [\ell]} R_j| < k$, then
    the $K_{\ell+1}$-partition
    $\{R_1, \dots, R_\ell, V(G)-\bigcup_{j\in[\ell]} R_j\}$ of $G$ satisfies \ref{item:twtau}, \ref{item:RV}, \ref{item:Hcomplete}. 
    Since $2 \leq c(h,d,k) \cdot t$ and $k\leq c(h,d,k) \cdot t$, item~\ref{item:pwx} also holds 
    and we are done.
    Now assume that $G-\bigcup_{j \in [\ell]} R_j$ has at least $k$ vertices. This enables us to enforce $\ell = k$, indeed, if $\ell<k$, then pick distinct vertices $s_{\ell+1}, \dots, s_k \in V(G)- \bigcup_{j\in[\ell]}R_j$ and set $R_j = \{s_j\}$ for every $j \in \{\ell+1, \dots, k\}$.
    From now on, we assume that $\ell=k$. 

    If $G-\bigcup_{j\in[k]} R_j$ is not connected, then for every component $C$ of $G-\bigcup_{j\in[k]} R_j$, 
    Apply induction to the instance 
    $(h,d,k,t,G[V(C) \cup \bigcup_{j\in[k]} R_j],\set{R_1,\ldots,R_{k}})$ 
    to obtain a graph $H^C$ with distinguished vertices $x^C_1, \dots, x^C_k$, and an $H^C$-partition $(V^C_x \mid x \in V(H^C))$ of $G[V(C) \cup \bigcup_{j \in [k]}R_j]$ satisfying \ref{item:pfirst}-\ref{item:plast}. 
    In particular, $V^C_{x^C_j} = R_j$ for every $j \in [k]$.
    Let $C_1, \dots, C_m$ be the components of $G-\bigcup_{j \in [k]} R_j$.
    Let $H$ be the graph obtained from the disjoint union of $H^{C_1}, \dots, H^{C_m}$ by identifying the vertices in $\{x^{C_j}_i\}_{j\in[m]}$ into a single vertex $x_i$, for each $i \in [k]$.
    Finally, set $V_{x_j} = R_j$ for every $j \in [k]$, and $V_x = V^{C_i}_x$ for every $x \in V(H^{C_i}) \setminus \{x^C_1, \dots x^C_k\}$ and for every $i \in [m]$.
    Item~\ref{item:twtau} holds since  
    $\tw(H) = \max_{i\in[m]}\{\tw(H^{C_i})\} \leq \tau(h,k)$.
    Item~\ref{item:pwx} holds by induction.
    Items~\ref{item:RV} and \ref{item:Hcomplete} hold by construction of $H$.
    From now on, we assume that $G-\bigcup_{j\in[k]} R_j$ is connected.

    Let $\mathcal{F}$ be the family of all connected subgraphs of $G - \bigcup_{j\in[k]} R_j$ containing an $\{R_1, \dots, R_k\}$-attached model of $K_{k+1}\oplus U_{h-1,d}$.
    If $\mathcal{F}$ contains $(d-1)2^k+1$ pairwise disjoint subgraphs, then 
    by the pigeonhole principle
    there exist $s_1,\ldots,s_k$ with $s_j\in R_j$ for each $j \in [k]$, and $d$ vertex-disjoint $\{s_1, \dots, s_k\}$-attached models 
    $\cgM^i = (M^i_x \mid x \in V(K_{k+1}\oplus U_{h-1,d}))$ of $K_{k+1}\oplus U_{h-1,d}$ for $i \in [d]$.
    We denote by $v_1,\dots, v_{k+1}$ the vertices of $K_{k+1}$ in $K_{k+1} \oplus U_{h-1,d}$. 
    Since these vertices have the same closed neighborhood in $K_{k+1}\oplus U_{h-1,d}$, we can assume that $M^{i}_{v_j}$ contains a neighbor of $s_j$, for all $i\in[d]$ and $j\in[k]$. 
    For each $j \in [k]$, let $M_j = \set{s_j} \cup\bigcup_{i\in[d]}M^i_{v_j}$.
    Note that for every $i \in [d]$, $\cgN^i = (M_x^i \mid x \in V(K_{k+1} \oplus U_{h-1,d}) - \{v_1,\dots,v_k\})$ is a model of $K_1 \oplus U_{h-1,d}$ in $G$.
    Moreover, for every $j \in [k]$, $i \in [d]$, and $M \in \cgN^i$, $M_j$ is adjacent to $M$.
    Therefore, $\cgN^1,\dots,\cgN^d$ together with $M_1,\ldots,M_k$ constitute a model of $K_{k}\oplus U_{h,d}$ in $G$, a contradiction (see \Cref{fig:combining:minors}).
    
    Hence, there are no $(d-1)2^k+1$ pairwise disjoint members in $\mathcal{F}$.
    Since $\tw(G-\bigcup_{j\in(k]} R_j) \leq \tw(G)<t$ and $G-\bigcup_{j \in [k]} R_j$ is connected, by \Cref{lemma:natural_tree_decomposition}, $G$ admits a natural tree-decomposition $\mathcal{W}$ of width at most $t-1$. By Lemma~\ref{lemma:helly_property_tree_decomposition}, there exists a set $Y$ of vertices of $G - \bigcup_{j\in[k]} R_j$ that is the union of at most $(d-1)2^k$ bags of $\mathcal{W}$, such that $Y$ intersects all the members of $\mathcal{F}$.
    By \Cref{lemma:increase_X_to_have_small_interfaces} applied to $G-\bigcup_{j \in[k]}R_j,\ \mathcal{W},\ Y$, there exists a set $X$ of at most $(2 (d-1)2^k - 1) \cdot t \leq c(h,d,k) \cdot t$ vertices in $G$ such that $Y\subseteq X$ and each component $D$ of $G-\bigcup_{j \in [k]} R_j - X$ has neighbors in at most two components of $G-\bigcup_{j=1}^kR_j - D$.

\begin{figure}[!htbp] 
   \centering 
   \includegraphics[scale=0.95]{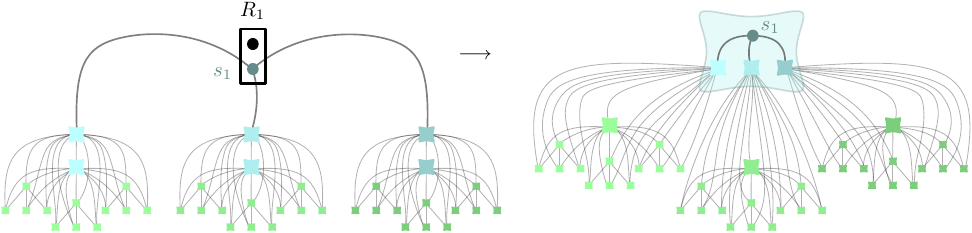} 
   \caption{Given three $\{s_1\}$-attached models of $K_2 \oplus U_{2,3}$, contract $s_1$ with the neighboring branch sets to obtain $K_1 \oplus U_{3,3}$.
   }
   \label{fig:combining:minors}
\end{figure}

    Consider the graph $G'$ obtained from $G-X$ by identifying the vertices in $R_j$ into a single vertex $r_j$, for each $j\in[k]$.
    Let $R = \{r_1, \dots, r_k\}$.
    Note that $G'$ is not necessarily a minor of $G$, however, $G'-R$ is a subgraph of $G$.  
    Observe that $G'$ has no $R$-attached model of $K_{k+1}\oplus U_{h-1,d}$ since $X$ is disjoint from $V(G')$ and every such model in $G$ intersects $X$.
    By \Cref{lemma:find_all_cuts_and_the_simple_part_in_the_middle} applied with $a=k+1$, we obtain an induced subgraph $C$ of $G'$ with the following properties.
    Let $C^1,\ldots, C^m$ be the connected components of $G'-C$, let $N^i=N_{G'}(V(C^i))$ for every $i\in [m]$, and let $C^0$ be the graph obtained from $C-R$ by adding all missing edges between vertices of $N^i-R$ for each $i\in[m]$. 
    Observe that: 
    \begin{enumerate}
    \item $R \subseteq V(C)$,
    \item $|N^i|\leq k-1$ for each $i \in [m]$,
    \item $C^0$ is $K_{2k+1}\oplus U_{h-1,d+2k}$-minor-free,
    \item $C^0$ is a minor of $G'-R$.
    \end{enumerate}
    In particular $C^0$ is a minor of $G$ and $\tw(C^0) < t$.
    
    If $h=1$, then $K_{2k+1}\oplus U_{h-1,d+2k}=K_{2k+1}$. 
    Moreover, since $(h,k) \neq (1,0)$, we have $k \geq 1$.
    So, we can apply~\Cref{theorem:illingworth_scott_wood_K_k_minor_free} to $C^0$, which is $K_{2k+1}$-minor-free, to obtain a graph $H^0$ with $\tw(H^0)\leq 2k-1=\tau(0,2k+1)$ and an $H^0$-partition $(V^0_x \mid x \in V(H^0))$ of $C^0$ of width at most $(d+2k)t = c(0,d+2k,2k+1) \cdot t$.
    If $h>1$, then apply induction to the instance $(h-1,d+2k,2k+1,t,C^0,\emptyset)$. 
    In both cases, we obtain a graph $H^0$ with $\tw(H^0)\leq\tau(h-1,2k+1)$ and an $H^0$-partition $(V^0_x \mid x \in V(H^0))$ of $C^0$ of width at most $c(h-1,d+2k,2k+1) \cdot t$.
    For every $i \in [m]$, the set $N^i-R$ is a clique in $C^0$. 
    Hence, the parts containing vertices in $N^i-R$ form a clique in $H^0$.

    Fix $i\in [m]$. 
    Let $D^i$ be the component of $G-\bigcup_{j=1}^k R_j - X$
    containing $C^i$.
    Let $X^i_1, \dots, X^i_q$ be the components of $G-\bigcup_{j=1}^k R_j - D^i$ with a neighbor in $D^i$.
    Note that $q \leq 2$ by the definition of $X$.

    Let $G^i$ be the graph obtained from $G$ as follows:
    \begin{enumerate}
        \item identify the vertices in $X^i_a$ into a single vertex $x^i_a$, for each $a \in [q]$,
        \item if $q > 0$ let 
        \begin{align*}
            \mathcal{R}^i &= \{R_j \mid j \in [k], r_j \in N^i\} \cup \{\{u\} \mid u \in N^i - R\}\cup\{\{x^i_a \mid a \in[q]\}\}\\
            \shortintertext{and if $q=0$, let}
            \mathcal{R}^i &= \{R_j \mid j \in [k], r_j \in N^i\} \cup \{\{u\} \mid u \in N^i - R\}.
        \end{align*}
        \item remove all vertices outside $V(C^i) \cup \bigcup \mathcal{R}^i$.
    \end{enumerate}

    Observe that $G^i$ is a minor of $G$. 
    This implies that $K_{k}\oplus U_{h,d}$ is not a minor $G^i$ and 
    $\tw(G^i) < t$.
    Since $N_{G}(V(C^i))-\bigcup_{j=1}^k R_j-X \subseteq V(D^i)$, $\mathcal{R}^i$ is a family of at most $|N^i|+1 \leq k$ pairwise disjoint non-empty sets, each of at most two vertices in $G^i$. Since $|N^i|\leq k-1$, there exists $j\in[k]$ such that $r_j\not\in N^i$ and therefore $R_j$ is disjoint from $V(G^i)$. 
    Thus, $|V(G^i)|<|V(G)|$. 

    Now apply induction to the instance
    $(h,d,k,t,G^i,\mathcal{R}^i)$.
    It follows that there is a graph $H^{i}$ with $\tw(H^{i})\leq\tau(h,k)$ and an $H^{i}$-partition $(V^{i}_x \mid x \in V(H^{i}))$ of $G^i$ 
    of width at most $c(h,d,k) \cdot t$ such that $\mathcal{R}^i = \{V_{x_{i,j}}^{i} \mid j\in[|\cgR^i|]\}$ for some clique $\{x_{i,1}, \dots,x_{i,|\mathcal{R}^i|}\}$ in $H^{i}$.

\begin{figure}[!htbp] 
   \centering 
   \includegraphics[scale=1]{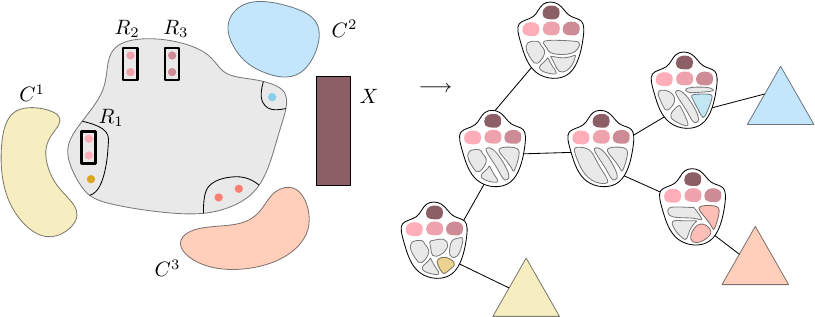} 
   \caption{After obtaining $X,C^0,C^1,\dots,C^m$, we call induction (on $h$) on $C^0$ (the grey region without the $R_i$s and with cliques attached to each $N^i$, see \Cref{fig:components:C}) to obtain $H^0$ and an $H^0$-partition along with a tree-decomposition.
   To $H^0$ we add $k+1$ dominant vertices $x_1, \dots, x_k,z$ corresponding to the parts $R_1, \dots, R_k$ and $X$.
   This gives the graph $\{x_1, \dots, x_k,z\} \oplus H^0$.
   To each bag of the tree-decomposition of $H^0$, we add parts corresponding to all $R_j$ and $X$. This yields a tree-decomposition of $\{x_1, \dots, x_k,z\} \oplus H^0$. 
   Next, we call induction (on the number of vertices) for each $C^i$ to obtain $H^i$ and an $H^i$-partition along with a tree-decomposition. To obtain the final result, that is, a graph $H$ with an $H$-partition, and a tree-decomposition of $H$, we
   perform a clique-sum between $\{x_1, \dots, x_k, z\} \oplus H^0$ and each $H^i$. This is possible since the parts corresponding to vertices in $N^i$ form a clique in $H^0$, and so have a common bag in the tree-decomposition of $H^0$.}
   \label{fig:final:clique:sums}
\end{figure}

    Finally, define the graph $H$ by the following process (see \Cref{fig:final:clique:sums} for an informal summary of the rest of the proof).
    Start with the disjoint union of $H^0$ and $H^1, \dots, H^m$.
    Add a clique
 $\{x_1, \dots, x_{k},z\}$   of  $k+1$ new vertices, each adjacent to every vertex of $H^0$.
    For every $i \in [m]$, 
    let $f_i$ be a mapping of $\{x_{i,1},\dots, x_{i,|\mathcal{R}^i|}\}$ to $V(H^0) \cup\{x_1, \dots, x_k,z\}$ defined as follows:
\[
f_i(x_{i,j}) = \begin{cases}
w& \textrm{if $w \in V(H^0)$ and $V^{i}_{x_{i,j}} \subseteq V^0_w$,}\\
x_{j'}& \textrm{if $j' \in [k]$ and $V^{i}_{x_{i,j}} = R_{j'}$,}\\
z& \textrm{if $V^{i}_{x_{i,j}} = \{x^i_a \mid a\in[q]\}$}
\end{cases}
\]
for each $j\in[|\mathcal{R}^i|]$.
    Now, identify $x_{i,j}$ with $f_i(x_{i,j})$ for every $i \in [m]$ and every $j \in [|\mathcal{R}_i|]$.
    This identification step can be seen as a result of the sequence of clique-sums between $\{x_1, \dots, x_k,z\}\oplus H^0$ and the graphs $H^{i}$ according to $f_i$ for $i\in[m]$.
    This completes the definition of $H$. Note that 
    \begin{align*}
    \tw(H) &\leq \max\left\{{\tw(\{x_1, \dots, x_k,z\}\oplus H^0)},\max_{i\in[m]}\tw(H^i)\right\}\\
    &\leq \max\left\{k+1+\tau(h-1,2k+1), \tau(h,k)\right\}\\
    &\leq \tau(h,k).
    \end{align*}
   
Define an $H$-partition $(V_x \mid x \in V(H))$ of $G$, where for each $x\in V(H)$,
    \[
    V_x = \begin{cases}
       R_j &\textrm{if $x=x_j$ for $j\in[k]$,} \\
       X &\textrm{if $x=z$,} \\
       V_x^0&\textrm{if $x \in V(H^0)$,} \\
       V_x^{i}&\textrm{if $x \in V(H^{i}) - \{x_{i,1}, \dots, x_{i,|\cgR^i|}\}$.}
    \end{cases}
    \]

    As mentioned $\tw(H)\leq \tau(h,k)$ so \cref{item:twtau} holds.
    For every $x \in V(H)$ distinct from $z$, $|V_x| \leq \max\{c(h-1,d+2k,2k+1)\cdot t,c(h,d,k)\cdot t\} = c(h,d,k) \cdot t$, and $|V_z| = |X| \leq (2(d-1)2^k-1) \cdot t \leq c(h,d,k) \cdot t$. Thus, \cref{item:pwx} holds.
    \Cref{item:RV} holds by the definition of the $H$-partition. 
    Finally, \cref{item:Hcomplete} holds by the definition of $H$.
\end{proof}

%% file: s.wcol_main.tex
First, we recall the definition of weak coloring numbers.
Given a graph $G$, a linear ordering $\sigma$ of $V(G)$, a vertex $v$ of $G$, and an integer $r\geq 1$, define $\WReach_r[G,\sigma,v]$ to be the set of vertices $w$ of $G$ such that there is a path from $v$ to $w$ of length at most $r$ whose minimum with respect to $\sigma$ is $w$.
Then define $\wcol_r(G,\sigma) = \max_{v \in V(G)}|\WReach_r[G,\sigma,v]|$
and $\wcol_r(G) = \min_\sigma \wcol_r(G,\sigma)$.

In this section, we prove the following theorem.

\weakcol*

Theorem~\ref{thm-wcol-main} follows immediately from the next result.

\begin{theorem}\label{theorem:bound_wcol_Uhd_minor_free}
There exist functions $f$ and $g$ such that
for all integers $h,d,r\geq 1$ and 
for every $U_{h,d}$-minor-free graph $G$,
\[
\wcol_r(G) \leq f(h,d)\cdot r^{g(h)}.
\]
\end{theorem}



The proof of~\Cref{theorem:bound_wcol_Uhd_minor_free} builds on a good understanding of the behavior of weak coloring numbers of graphs excluding a complete graph as a minor, and also of graphs of bounded treewidth.
Van den Heuvel, Ossona de Mendez, Quiroz, Rabinovich, Siebertz~\cite{vdHetal17} proved that 
for every integer $t\geq4$, 
every $K_t$-minor-free graph $G$ satisfies
$\wcol_r(G)\leq \binom{r+t-2}{t-2}(t-3)(2r+1)$ for every integer $r\geq 1$.
Their proof technique, specifically chordal partitions of graphs, inspired a lot of follow-up research, including our work on weak coloring numbers.
The base case of our main technical contribution (\Cref{lemma:structure_for_wcol}) relies on the following structural result from~\cite{vdHetal17} that underlies the upper bound on weak coloring numbers for $K_t$-minor-free graphs. We include a rough sketch of the proof -- see \Cref{fig:wcol_Kt_minor_free}.
Recall that a geodesic in a graph is a shortest path between its endpoints. 
\begin{lemma}[Lemma 4.1 in~\cite{vdHetal17}]
\label{lemma:Heuvel_et_al_wcol_Kt_minor_free}
    Let $t\geq 3$
    and let $G$ be a $K_t$-minor-free graph.%
    \footnote{
    The statement in~\cite{vdHetal17} assume $t \geq 4$ and item~\ref{item:heuvel_et_al_wcol_Kt_minor_free_number_of_geodesics} bounds the number of geodesics by $t-3$. However the statement also holds for $t=3$ with $t-3$ replaced by $\max\{t-3,1\}$ in item~\ref{item:heuvel_et_al_wcol_Kt_minor_free_number_of_geodesics}.
    }
    Then there is a graph $H$ and an $H$-partition $(V_x \mid x \in V(H))$ of $G$ together with an ordering $x_1, \dots, x_{|V(H)|}$ of $V(H)$ such that 
    \begin{enumerate}
        \item $G[V_i]$ is connected for every $i \in [|V(H)|]$, in particular, $H$ is a minor of $G$;
        \item $\{x_j \mid j < i \text{ and } x_jx_i \in E(H)\}$ is a clique in $H$, for every $i \in [|V(H)|]$;
        \item $\tw(H) \leq t-2$;
        \item $V_{x_i}$ is the union of the vertex sets of at most $\max\{t-3,1\}$ geodesics in $G[V_{x_i} \cup \dots \cup V_{x_{|V(H)|}}]$, for every $i \in [|V(H)|]$.\label{item:heuvel_et_al_wcol_Kt_minor_free_number_of_geodesics}
    \end{enumerate}
\end{lemma}

\begin{figure}[!htbp] 
   \centering 
   \includegraphics[scale=1]{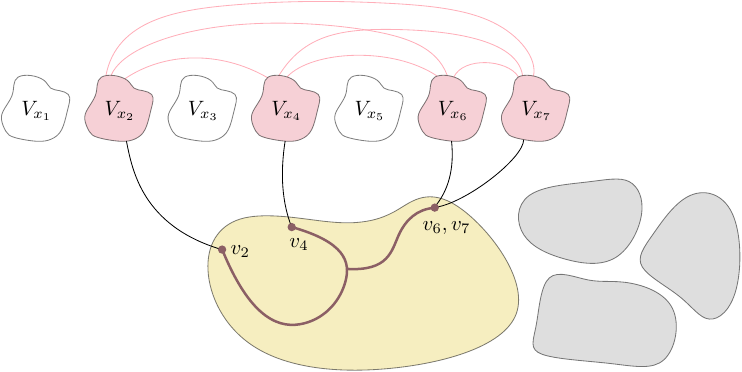} 
   \caption{We sketch the proof of \Cref{lemma:Heuvel_et_al_wcol_Kt_minor_free} with an illustration. The construction of $H$ is an iterative inductive procedure, where the items in the statement are invariants. 
   We illustrate an inductive step using an example in the figure. 
   Suppose that parts $V_{x_1}, \dots, V_{x_7}$
   are already constructed. 
   Pick any component of the remainder of the graph, say the yellow one. 
   The yellow component is adjacent to some parts: $V_{x_2}, V_{x_4}, V_{x_6}, V_{x_7}$. By the invariant, the parts form a clique in $H$, and so, we have a model of $K_4$.
   Thus together with the yellow component, we have a model of $K_5$. 
   Recall that in the general setting, the graph is $K_t$-minor-free, thus, the number of parts that the yellow component is adjacent to is bounded by $t-1$. 
   For each $i \in \{2,4,6,7\}$, fix a vertex $v_i$ in the yellow component adjacent to $V_{x_i}$ (note that these vertices are not necessarily pairwise distinct). 
   Now, connect the vertices $v_i$ by geodesics in the yellow part, it suffices to take $t-2$ of them. Finally, take the vertices $v_i$ and the geodesics as the new part $V_{x_8}$. 
   It is easy to verify that the invariant is preserved.}
   \label{fig:wcol_Kt_minor_free}
\end{figure}

Let $t$ and $r$ be integers with $t,r\geq0$. 
For every graph $G$ with $\tw(G)\leq t$, we have $\wcol_r(G)\leq \binom{r+t}{t}$, 
as proved by Grohe, Kreutzer, Rabinovich, Siebertz, and Stavropoulos~\cite{Grohe15}. We will need the following slightly more precise statement that follows line-by-line from the proof in~\cite{Grohe15}.

\begin{lemma}
[Theorem 4.2 in \cite{Grohe15}]
\label{theorem:Grohe_et_al_wcol_bounded treewidth}
    Let $t$ and $r$ be integers with $t,r\geq0$.
    Let $G$ be a graph and let
    $\sigma=(x_1, \dots, x_{|V(G)|})$ 
    be an ordering of $V(G)$ such that for every $i \in [|V(G)|]$, the set 
    $\{x_j  \mid j<i \textrm{ and } x_i x_j \in E(G)\}$ is a clique of size at most $t$ in $G$.
    Then
    \[
        \wcol_r(G,\sigma) \leq \binom{r+t}{t}.
    \]
\end{lemma}

Note that the above two lemmas easily imply the mentioned bound on $\wcol_r(G)$ for $K_t$-minor-free graphs. To see this, order the vertices according to the index of a part of the $H$-partition that they are in, and within each part, we order the vertices arbitrarily. Now, to verify the bound, we need a simple observation on geodesics that we will prove later, see \Cref{lemma:intersection_ball_with_geodesic}.

If a graph $G$ has bounded treewidth, say $\tw(G)<t$, then $G$ satisfies 
the Helly property articulated by~\Cref{lemma:helly_property_tree_decomposition}. 
Namely, when $\cgF$ is a family of connected subgraphs of $G$, then either there are $d$ pairwise disjoint members of $\cgF$, or there is subset of $(d-1)t$ vertices of $G$ hitting all members of $\cgF$.
One of the main difficulties that arises when trying to prove~\Cref{theorem:bound_wcol_Uhd_minor_free} is to find an equally useful statement as \Cref{lemma:helly_property_tree_decomposition}, but for $K_t$-minor-free graphs.
This is the motivation for~\Cref{lem:X:geodesics}. We defer the proof of it to~\Cref{sec:graph-minor-struggles}.

\begin{restatable}{lemma}{geodesics}
\label{lem:X:geodesics}
    There exist functions $\gamma$, $\delta$ such that for all integers $t,d \geq 1$, for every $K_t$-minor-free graph $G$, for every family $\cgF$ of connected subgraphs of $G$ either:
    \begin{enumerate}
        \item there are $d$ pairwise vertex-disjoint subgraphs in $\mathcal{F}$, or \label{lem:item:X:geodesics:disjoint}
        \item there exist $A \subseteq V(G)$ such that $|A| \leq (d-1) \gamma(t)$, and a subgraph $X$ of $G$, where
        $X$ is the union of at most $(d-1)^2\delta(t)$ geodesics in $G-A$,\label{lem:item:X:geodesics:union} and for every $F \in \cgF$ we have $V(F) \cap (V(X) \cup A) \neq \emptyset$.\label{lem:item:X:geodesics:hit}
    \end{enumerate}    
\end{restatable}

With~\Cref{lem:X:geodesics} in hand, 
we are ready to proceed with~\Cref{lemma:structure_for_wcol}, the key technical contribution standing behind~\Cref{theorem:bound_wcol_Uhd_minor_free}.
The proof relies on some ideas from the proof of \Cref{lemma:main_induction}, see \Cref{fig:wcol_general} for a sketch of the proof.
After the proof of~\Cref{lemma:structure_for_wcol} we complete the final argument for~\Cref{theorem:bound_wcol_Uhd_minor_free}.
Recall the definition of $\tau : \mathbb{Z}^2_{\geq 0} \rightarrow \mathbb{Z}$:
\begin{align*}
    \tau(0,k) &= k-2, \textrm{ and}\\ 
    \tau(h,k) &= \tau(h-1,2k+1)+k+1\ \textrm{for every $h\geq 1$}.
\end{align*}

Let $t(h,d,k)$ be the number of vertices in $K_{k}\oplus U_{h,d}$; that is, for all $h,d\geq1$ and $k\geq0$,
\[
t(h,d,k)= k+d(d^{h}-1)/(d-1). 
\]
Let $\epsilon: \mathbb{Z}_{\geq 0}\times \mathbb{Z}_{>0}^2 \rightarrow \mathbb{Z}$ be the function defined by 
\begin{align*}
    \epsilon(0,d,k)&= \max\{k-3,\ 1\},\ \text{and}\\
    \epsilon(h,d,k)&= \max\{d-1,\ k,\ (d-1)\gamma(t(h,d,k))+2(d-1)^2\delta(t(h,d,k))-1, \\ 
    &\qquad \qquad \epsilon(h-1,d+2k,2k+1)\},\ \textrm{for every $h \geq 1$}.
\end{align*}

A set $S$ of vertices in a graph $G$ 
is a \emph{subgeodesic} in $G$ 
if there is a supergraph $G^+$ of $G$ and a geodesic $P$ in $G^+$
such that $S \subseteq V(P)$.

\begin{figure}[!htbp] 
   \centering 
   \includegraphics[scale=1]{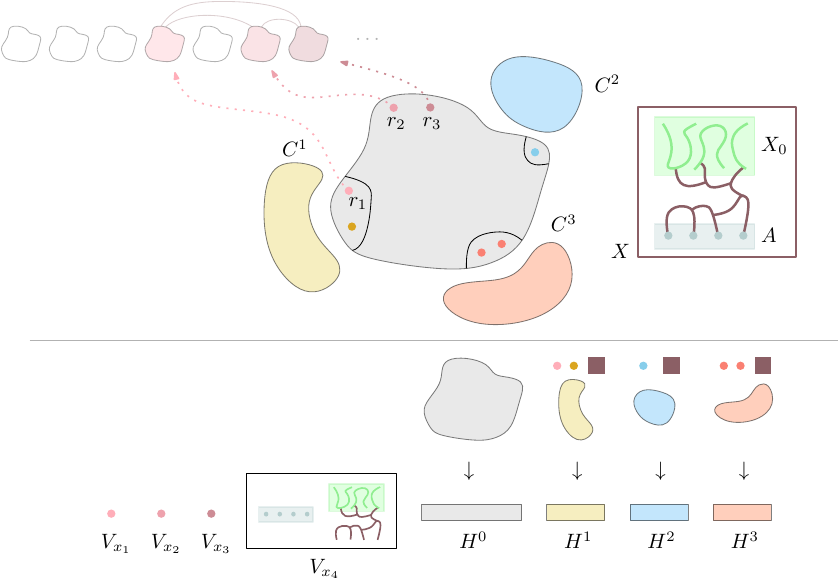} 
   \caption{
   Overview of the proof of \Cref{lemma:structure_for_wcol}.
   The vertices $r_i$ (we usually call them the interface) 
   correspond to some already processed parts, that form a clique in $H$ -- this follows the idea explained in the caption of \Cref{fig:wcol_Kt_minor_free}. 
   The first part of the proof is similar to the proof of \Cref{lemma:main_induction}. 
   That is, we set $\mathcal{F}$ to be the family of all connected subgraphs of $G-R$ containing an $R$-attached model of $K_{k+1} \oplus U_{h-1,d}$. 
   Again, one can prove that $\mathcal{F}$ does not contain $d$ pairwise disjoint elements, as otherwise, we could build a model of $K_k \oplus U_{h,d}$ in $G$ (see \Cref{fig:combining:minors}).
   By \Cref{lem:X:geodesics}, there exists a hitting set for $\mathcal{F}$, that consists of two parts, the first one -- $A$ -- is not too big, and the second one -- $X_0$ -- is a union of a small number of geodesics. We can make $X_0 \cup A$ connected by adding a few more geodesics in $G-R-A$.
   Thus we obtain a hitting set $X$ for $\mathcal{F}$ that will be a new part of the partition.
   Now, we can apply the decomposition lemma to $G-X$ (\Cref{lemma:find_all_cuts_and_the_simple_part_in_the_middle}).
   The grey fragment of the graph in the figure is $K_{2k+1}\oplus U_{h-1,d+2k}$-minor-free, thus, we can process it by induction on $h$ to obtain a graph $H^0$ and an $H^0$-partition of this fragment.
   Then we process the components $C^i$ by induction on the number of vertices, with the interface consisting of the vertices in $N^i$ (possibly containing some of the $r_i$s), and a vertex corresponding to the set $X$ contracted.
   The graph $H$ is obtained exactly as in the proof of \Cref{lemma:main_induction} (see \Cref{fig:final:clique:sums}).
   At the bottom of the figure, we depict the final ordering of the vertices of $H$.
   First, we put vertices corresponding to $r_i$s.
   The set $X$ becomes a new part in $H$, however, note that in the final ordering for the weak coloring number, first, we put the vertices in $A$, and then the vertices in $X_0$.
   Then comes the vertices of $H^0$ (with the ordering obtained by induction), and finally the vertices of $H^i$ (ordered again by induction) for every $i$.
   }
   \label{fig:wcol_general}
\end{figure}

\begin{lemma}\label{lemma:structure_for_wcol}
    For all integers $h,d,k,\ell$ 
    with $h,d\geq1$ and $k\geq \ell\geq 0$,  
    for every graph $G$ such that $K_{k}\oplus U_{h,d}$ is not a minor of $G$, for all pairwise distinct vertices $r_1,\ldots,r_\ell$  in $G$,
    there is a graph $H$ with an ordering $x_1, \dots, x_{|V(H)|}$ of $V(H)$, and an $H$-partition $(V_x \mid x \in V(H))$ of $V(G)$ such that:
    \begin{enumerateNum}
        \item $\{x_j\mid j<i \textrm{ and $x_j x_i\in E(H)$}\}$ is a clique in $H$ for all $i\in[|V(H)|]$; \label{item-p-first}\label{item-elimination-ordering}
        \item $\{x_1, \dots, x_\ell\}$ is a clique in $H$; \label{item-x1-x_l_is_a_clique}
        \item $\tw(H) \leq \tau(h,k)$;\label{item:structure_for_wcol:twH}
        \item $V_{x_j} = \{r_j\}$ for all $j \in [\ell]$;\label{item-Rj_is_a_part}
        \item\label{item-partition_into_geodesics} for each integer $i$ with $\ell+1\leq i \leq |V(H)|$, there exists a partition $(A_{x_i},B_{x_i})$ of $V_{x_i}$ such that:
        \begin{enumerate}
            \item $|A_{x_i}| \leq \epsilon(h,d,k)$, and
            \item $B_{x_i}$ is the union of at most $\epsilon(h,d,k)$ subgeodesics in $G[B_{x_i} \cup \bigcup_{j > i} V_{x_j}]$.
        \end{enumerate} \label{item-p-last}
    \end{enumerateNum}
\end{lemma}

\begin{proof}
    We call a tuple 
    $(h,d,k,G,\set{r_1,\ldots,r_\ell})$ satisfying the premise of the lemma an \emph{instance}.
    We proceed by induction on $(h,|V(G)|)$ in lexicographic order.
    Let $R = \{r_1, \dots, r_\ell\}$.

    If $h=1$ and $k=0$, then $K_k \oplus U_{h,d}$ is the graph with $d$ vertices and no edges.
    Thus $|V(G)| \leq d-1$, and $\{V(G)\}$ is a $K_1$-partition of $G$ of width at most $d-1$.
    Then take $\sigma = (x)$ where $x$ is the vertex of $K_1$.
    Items~\ref{item-elimination-ordering}, \ref{item-x1-x_l_is_a_clique} are clear. \Cref{item:structure_for_wcol:twH} holds as $\tau(1,0) = 0$.
    \Cref{item-Rj_is_a_part} holds vacuously.
    Finally, \cref{item-partition_into_geodesics} holds by taking $A_x = V(G)$
    and $B_x = \emptyset$, since $d-1 \leq \epsilon(1,d,0)$.
    Now we assume $(h,k) \neq (1,0)$.

    If $|V(G)-R| \leq k$, then 
    take $H = K_{\ell+1}$ with vertex set $\{x_1, \dots, x_{\ell+1}\}$. 
    Let $V_{x_j} = \{r_j\}$ for every $j \in [\ell]$ and
    let $V_{x_{\ell+1}} = V(G)-R$.
    Note that $(V_x \mid x \in \{x_1, \dots, x_{\ell+1}\})$ is a $K_{\ell+1}$-partition of $G$.
    Let $A_{x_{\ell+1}} = V(G)-R$ and $B_{x_{\ell+1}} = \emptyset$.
    In particular, $|A_{x_{\ell +1}}| \leq \epsi(h,d,k)$. 
    It is straightforward to check that \ref{item-p-first}-\ref{item-p-last} hold.
    Now, if $|V(G)-R| > k$ and $\ell<k$, then set $r_{\ell+1}, \dots, r_k$ to be distinct vertices of $G - R$.
    Therefore, from now on assume 
    $\ell=k$ and $V(G)- R$ is non-empty.

    Suppose that $G-R$ is disconnected.
    Let $C^1, \dots, C^m$ be the components of $G-R$.
    For every $i \in[m]$ 
    Apply induction to the instance 
    $(h,d,k,G[V(C^i) \cup R],\set{r_1,\ldots,r_{k}})$ 
    and obtain $H^i$ with $V(H^i)=\{x^i_1, \dots, x^i_{|V(H^i)|}\}$ and an $H^i$-partition $(V^i_x \mid x \in V(H^i))$ of $G[V(C^i) \cup R]$ satisfying \ref{item-p-first}-\ref{item-p-last}. 
    Let $H$ be the graph obtained from the disjoint union of $H^{1}, \dots, ,H^{m}$ by identifying the vertices in $\{x^{i}_j\}_{i\in[m]}$ into a single vertex $x_j$, for each $j \in [k]$.
    Then order the vertices of $H$ by
    \[
    \sigma= (x_1, \dots, x_k, x^{1}_{k+1}, \dots, x^{1}_{|V(H^{1})|}, \dots,x^m_{k+1}, \dots, x^{m}_{|V(H^{m})|}).
    \]
    Finally, set $V_{x_j} = \{r_j\}$ for each $j \in [k]$, and $V_x = V^{i}_x$ for every $x \in V(H^{i}) \setminus \{x^{i}_1,\dots,x^{i}_k\}$ for each $i\in[m]$.
    Items~\ref{item-x1-x_l_is_a_clique} and \ref{item-Rj_is_a_part} follow by construction of $H$ and $(V_x \mid x\in V(H))$. 
    In order to prove item~\ref{item-elimination-ordering}, consider $x \in V(H)$ and 
    let $N$ be the neighbors of $x$ in $H$ that are smaller than $x$ in $\sigma$. 
    If $x \in\{x_1,\dots, x_k\}$, then clearly $N$ is a clique in $H$. 
    If $x \in V(H^i)$ for some $i \in[m]$, let $Y = N \cap \{x_1, \dots, x_k\}$ and $Z = N - Y$.
    Observe that $Z \subseteq V(H^i) \setminus \{x^i_1, \dots, x^i_k\}$.
    Then by induction $\{x^i_j \mid j\in[k], x_j \in Y\} \cup Z \subseteq V(H^i)$ is a clique in $H^i$ and so $N$ is a clique in $H$. 
    This proves item~\ref{item-elimination-ordering}.
    Note that $\tw(H)=\max_{i\in[m]}\{\tw(H^i)\}\leq \tau(h,k)$ which proves item~\ref{item:structure_for_wcol:twH}. 
    In order to prove item~\ref{item-partition_into_geodesics}, consider 
    $x^i_a$ for some $i \in [m]$ and $a \in [|V(H^i)|]$.
    Then by induction there exists a partition $A_{x^i_a},B_{x^i_a}$ of $V^i_{x^i_a}$ such that $|A_{x^i_a}| \leq \epsilon(h,d,k)$ and $B_{x^i_a}$ is the union of at most $\epsilon(h,d,k)$ subgeodesics in $G\left[B_{x^i_a} \cup \bigcup_{b>a} V^i_{x^i_b}\right]$.
    But since components of 
    $G\left[B_{x^i_a} \cup \bigcup_{b>a} V^i_{x^i_b}\right]$ are components of 
    $G\left[B_{x^i_a} \cup \bigcup_{y>_\sigma x^i_a} V_{y}\right]$, we deduce that $B_{x^i_a}$ is the union of at most $\epsilon(h,d,k)$ subgeodesics in 
    $G\left[B_{x^i_a} \cup \bigcup_{y>_\sigma x^i_a} V_{y}\right]$.
    This proves item~\ref{item-partition_into_geodesics}.

    Now assume that $G-R$ is connected.

    Let $\mathcal{F}$ be the family of all connected subgraphs of $G - R$ containing an $R$-attached model of $K_{k+1}\oplus U_{h-1,d}$.
    If $\mathcal{F}$ contains $d$ pairwise vertex-disjoint subgraphs, then 
    there exist $d$ vertex-disjoint $R$-attached models 
    $\cgM^i = (M^i_x \mid x \in V(K_{k+1} \oplus U_{h-1,d}))$ in $G$ for each $i \in [d]$.
    Denote by $v_1,\dots, v_{k+1}$ the vertices of $K_{k+1}$ in $K_{k+1}\oplus U_{h-1,d}$ and since these vertices are twins in $K_{k+1} \oplus U_{h-1,d}$, we can assume that $M^{i}_{v_j}$ contains a neighbor of $r_j$, for all $i\in[d]$ and $j\in[k]$.
    For each $j \in [k]$, let $M_j = \set{r_j} \cup\bigcup_{i\in[d]}M^i_{v_j}$.
    Note that for every $i \in [d]$, $\cgN^i = (M_x^i \mid x \in V(K_{k+1} \oplus U_{h-1,d}) - \{v_1,\dots,v_k\})$ is a model of $K_1 \oplus U_{h-1,d}$ in $G$.
    Moreover, for every $j \in [k]$, $i \in [d]$, and $M \in \cgN^i$, $M_j$ is adjacent to $M$.
    Therefore, $\cgN^1,\dots,\cgN^d$ together with $M_1,\ldots,M_k$ constitute a model of $K_{k}\oplus U_{h,d}$ in $G$, a contradiction.
    Hence, there are no $d$ pairwise disjoint members in $\mathcal{F}$.

    Let $t = |V(K_{k} \oplus U_{h,d})| = t(h,d,k)$. Note that $G$ is $K_t$-minor-free.
    By \Cref{lem:X:geodesics}, there is a set $A$ of at most $(d-1)\gamma(t)$ vertices in $G- R$, and a set $X_0$ which is the union of the vertex sets of at most $(d-1)^2\delta(t)$ geodesics in $G - R - A$,
    such that $A \cup X_0$ intersects every member of $\mathcal{F}$.
    If $A \cup X_0 = \emptyset$, then take $A = \emptyset$ and $X_0$ an arbitrary singleton included in $V(G)-R$.
    Since $G-R$ is connected, 
    we can add to $A\cup X_0$  at most $|A|+(d-1)^2\delta(t)-1$ geodesics in $G-R$ to obtain a set 
    $X$ such that $G[X]$ is connected.
    Let $B = X \setminus A$. Note that $B$ is the union of at most $(d-1)\gamma(t) + 2(d-1)^2\delta(t) - 1 \leq \epsilon(h,d,k)$ subgeodesics in $G - R - A$.

    By construction, $G[X]$ is connected and $G-X$ does not contain an $R$-attached model of $K_{k+1}\oplus U_{h-1,d}$.
    By \Cref{lemma:find_all_cuts_and_the_simple_part_in_the_middle} applied for $a=k+1$, we obtain an induced subgraph $C$ of $G-X$ with the following properties.
    Let $C^1,\ldots C^m$ be the connected components of $G-X-C$, let $N^i=N_{G-X}(V(C^i))$ for every $i\in [m]$, and let $C^0$ be the graph obtained from $C-R$ by adding all missing edges between vertices of $N^i-R$ for each $i\in[m]$. 
    Then 
    \begin{enumerate}
    \item $R \subseteq V(C)$,
    \item $|N^i|\leq k-1$ for each $i \in [m]$,
    \item $C^0$ is $K_{2k+1}\oplus U_{h-1,d+2k}$-minor-free,
    \item $C^0$ is a minor of $G-X-R$.
    \end{enumerate}

    If $h=1$, then $K_{2k+1}\oplus U_{h-1,d+2k}=K_{2k+1}$. 
    Moreover $k\geq 1$ since $(h,k) \neq (1,0)$, and so
    $2k+1 \geq 3$.
    Thus we can apply \Cref{lemma:Heuvel_et_al_wcol_Kt_minor_free} to $C^0$, which is $K_{2k+1}$-minor-free, and obtain a graph $H^0$ with $\tw(H^0)\leq 2k-1 =\tau(0,2k+1)$ and an $H^0$-partition $(V^0_x \mid x \in V(H^0))$ with an ordering $x_{0,1}, \dots, x_{0,|V(H^0)|}$ of $V(H^0)$ such that for every $p \in [|V(H^0)|]$, $V^0_{x_{0,p}}$ is the union of at most $\max\{2k-2,1\} \leq \epsilon(0,d+2k,2k+1)$ geodesics in $C^0\left[V^0_{x_{0,p}} \cup \dots \cup  V^0_{x_{0,|V(H^0)|}}\right]$.
    Then set $A^0_{x_{0,i}} = \emptyset$ and $B^0_{x_{0,i}}=V^0_{x_{0,i}}$ for every $i \in [|V(H^0)|]$.
    If $h>1$, then apply induction to the instance $(h-1,d+2k,2k+1,C^0,\emptyset)$.

    In both cases, we obtain a graph $H^0$ with $\tw(H^0)\leq\tau(h-1,2k+1)$ and an $H^0$-partition $(V^0_x \mid x \in V(H^0))$ of $C^0$ with an ordering $\sigma^0 = (x_{0,1}, \dots, x_{0,|V(H^0)|})$ of $V(H^0)$ such that for every $p \in [|V(H^0)|]$, $V^0_{x_{0,p}}$ has a partition $(A^0_{x_{0,p}},B^0_{x_{0,p}})$ such that $|A^0_{x_{0,p}}| \leq \epsilon(h-1,d+2k,2k+1) \leq \epsilon(h,d,k)$ and $B^0_{x_{0,p}}$ is the union of at most $\epsilon(h-1,d+2k,2k+1)\leq \epsilon(h,d,k)$ subgeodesics in $C^0\left[B^0_{x_{0,p}} \cup \bigcup_{q>p}V^0_{x_{0,q}}\right]$. 
    For every $i \in [m]$, the graph $N^i-R$ is a clique in $C^0$. Hence, the parts containing vertices in $N^i-R$ form a clique in $H^0$.

    Fix some $i \in [m]$.
    Let $G^i$ be the graph obtained from $G[V(C^i) \cup N^i \cup X]$
    by contracting $X$ into a single vertex $z^i$.
    Note that $G^i$ is a minor of $G$ and therefore 
    $G^i$ has no model of $K_{k}\oplus U_{h,d}$.
    Since $|N^i|\leq k-1$, there exists $j\in[k]$ such that $r_j\not\in N^i$ and therefore $r_j\not\in V(G^i)$. 
    Thus, $|V(G^i)|<|V(G)|$. 
    Let $R^i = N^i \cup \{z^i\}$, so $|R^i|\leq k-1+1 = k$.
    Now, apply induction to the instance $(h,d,k,G^i,R^i)$.
    It follows that there is a graph $H^{i}$ with $\tw(H^{i})\leq\tau(h,k)$ and 
    an $H^{i}$-partition $(V^{i}_x \mid x \in V(H^{i}))$ of $G^i$ 
    and an ordering $\sigma_i = (x_{i,p})_{p \in [|V(H^i)|]}$ of $V(H^i)$
    such that for each $j\in [|R^i|]$ the set $V^i_{x_{i,j}}$ is a singleton, 
    $\bigcup_{j\in[|R^i|]} V^i_{x_{i,j}}=R^i$, the set 
    $\{x_{i,1}, \dots,x_{i,|R^i|}\}$ is a clique in $H^{i}$, and for every integer $p$ with 
    $|R^i| < p \leq |V(H^i)|$, 
    the set $V^i_{x_{i,p}}$ has a partition $(A^i_{x_{i,p}},B^i_{x_{i,p}})$ such that $|A^i_{x_{i,p}}| \leq \epsilon(h,d,k)$ and $B^i_{x_{i,p}}$ is the union of at most $\epsilon(h,d,k)$ subgeodesics in $G^i\left[B^i_{x_{i,p}} \cup \bigcup_{q>p} V^i_{x_{i,q}}\right]$.

    Finally, define the graph $H$ as follows. 
    Start with the disjoint union of $H^0$ and $H^1, \dots, H_m$.
    Add a clique $\{x_1, \dots, x_{k},z\}$ of $k+1$ new vertices, each adjacent to every vertex of $H^0$.
    For every $i \in [m]$, 
    let $f_i$ be a mapping of $\{x_{i,1},\dots, x_{i,|R^i|}\}$ to $V(H^0) \cup\{x_1, \dots, x_k,z\}$ defined as follows:
\[
f_i(x_{i,j}) = \begin{cases}
w& \textrm{if $w \in V(H^0)$ and $V^{i}_{x_{i,j}} \subseteq V^0_w$,}\\
x_{j'}& \textrm{if $j' \in [k]$ and $V^{i}_{x_{i,j}} = \{r_{j'}\}$,}\\
z& \textrm{if $V^{i}_{x_{i,j}} =\set{z^i}$,}
\end{cases}
\]
    for each $j \in[|R^i|]$.
    Now, identify $x_{i,j}$ with $f_i(x_{i,j})$ for every $i \in [m]$ and every $j \in [|R_i|]$.
    This identification step can be seen as a result of the sequence of clique-sums between $\{x_1, \dots, x_k,z\}\oplus H^0$ and the graphs $H^{i}$ according to $f_i$ for $i\in[m]$.
    This completes the definition of $H$. 
    
    Note that 
    \begin{align*}
    \tw(H) &\leq \max\left\{{\tw(\{x_1, \dots, x_k,z\}\oplus H^0)},\max_{i\in[m]}\tw(H^i)\right\}\\
    &\leq \max\left\{k+1+\tau(h-1,2k+1), \tau(h,k)\right\}\\
    &=\tau(h,k).
    \end{align*}
    
    Now define an $H$-partition $(V_x \mid x \in V(H))$ of $G$, where for each $x \in V(H)$,
\[
    V_x = \begin{cases}
       \set{r_j} &\textrm{if $x=x_j$ for $j\in[k]$,} \\
       X         &\textrm{if $x=z$,} \\
       V_x^0     &\textrm{if $x \in V(H^0)$,} \\
       V_x^{i}   &\textrm{if $x \in V(H^{i}) - \{x_{i,1}, \dots, x_{i,|R^i|}\}$}.
    \end{cases}
    \]

    Moreover, order the vertices of $H$ by
    \[
    \sigma=(x_1, \dots, x_k, z, x_{0,1}, \dots, x_{0,|V(H^0)|},
    \dots, x_{m,|R^m|+1},\dots, x_{m,|V(H^m)|}).
    \]

    In order to prove item~\ref{item-elimination-ordering}, consider a vertex $x \in V(H)$, and let $N = \{y \in V(H) \mid y <_\sigma x, xy \in E(H)\}$.
    If $x \in \{x_1, \dots, x_k,z\}$, then $N \subseteq \{x_1, \dots, x_k,z\}$
    and so $N$ is a clique in $H$.
    If $x \in V(H^0)$, then $N \setminus \{x_1, \dots, x_k,z\}$ is a clique in $H^0$, thus $N$ is a clique in $\{x_1, \dots, x_k,z\} \oplus H^0$, and so in $H$.
    If $x \in V(H^i)-\{x_{i,1}, \dots, x_{1,|R^i|}\}$ for some $i \in [m]$,
    let $N' = N \cap V(H^0)$ and $N'' = N \setminus N'$.
    Then $f_i^{-1}(N') \cup N'' = \{y \in V(H^i) \mid y <_{\sigma_i} x, xy \in E(H^i)\}$ is a clique in $H^i$,
    and so $N$ is a clique in $H$.
    This proves item~\ref{item-elimination-ordering}.

    \Cref{item-x1-x_l_is_a_clique} follows from the definition of $H$.
    As mentioned before $\tw(H)\leq \tau(h,k)$ so item \ref{item:structure_for_wcol:twH} holds.
    \Cref{item-Rj_is_a_part} follows from the definition of $(V_x \mid x \in V(H))$.
    For~\cref{item-partition_into_geodesics}, for each $x \in V(H) - \{x_1, \dots, x_k\}$, define
    \[
    A_x,B_x = \begin{cases}
    A,B&\textrm{if $x=z$},\\
    A^0_{x},B^0_{x}&\textrm{if $x\in V(H^0)$,}\\
    A^i_{x},B^i_{x}&\textrm{if $x\in V(H^i)-\set{x_{i,1},\ldots,x_{i,|R^i|}}$ for $i\in[m]$}.
    \end{cases}
    \]
    Consider now some $x \in V(H)-\{x_1, \dots, x_k\}$.
    First observe that $|A_x| \leq \epsilon(h,d,k)$.
    It remains to show that $B_x$ is the union of at most $\epsilon(h,d,k)$ subgeodesics in $G\left[B_x \cup \bigcup_{y >_\sigma x} V_y\right]$.
    If $x = z$, this follows from the definition of $A$ and $B$.
    If $x \in V(H^0)$, then there is a supergraph $C^+$ of $C^0$ such that $B_x=B_x^0$ is in the union of the vertex sets of at most $\epsilon(h,d,k)$ geodesics in $C^+[B^0_x \cup \bigcup_{y >_{\sigma^0}x} V^0_y]$.
    Let $C^{++}$ be obtained from the disjoint union of $C^+$ and $C^1, \dots, C^m$ by adding every edge between $V(C^i)$ and $V(C^0)$ that is in $G$, for each $i\in[m]$.
    Since $N^i \cap V(C^0)$ is a clique in $C^0$ for every $i\in[m]$, for every two vertices $u$, $v$ in $C^+$,  $\dist_{C^+}(u,v) = \dist_{C^{++}}(u,v)$.
    Hence $B_x$ is the union of the vertex sets of at most $\epsilon(h,d,k)$ geodesic in $C^{++}$, which is a supergraph of $G\left[B_x \cup \bigcup_{y >_\sigma X} V_y\right]$.
    This shows that $B_x$ is the union of at most $\epsilon(h,d,k)$ subgeodesics in $G\left[B_x \cup \bigcup_{y >_\sigma X} V_y\right]$, as desired.
    Finally, if $x \in V(H^i) \setminus\{x_{i,1}, \dots, x_{i,|R^i|}\}$, then
    $B_x=B^i_x$ is the union of at most $\epsilon(h,d,k)$ subgeodesics in $C^i\left[B^i_x \cup \bigcup_{y>_{\sigma^i} x} V^i_y\right]$.
    Since components of $C^i\left[B^i_x \cup \bigcup_{y>_{\sigma^i} x} V^i_y\right]$ are components of $G\left[B_x \cup \bigcup_{y>_{\sigma} x} V_y\right]$, we deduce that $B_x$ is the union of at most $\epsilon(h,d,k)$ geodesics in $G\left[B_x \cup \bigcup_{y>_{\sigma} x} V_y\right]$. This proves \cref{item-partition_into_geodesics} and concludes the proof.
\end{proof}

\begin{lemma}
\label{lemma:intersection_ball_with_geodesic}
    Let $G$ be a graph and let $r$ be a non-negative integer. 
    For every subgeodesic $S$ in $G$ and for every vertex $v \in V(G)$,
    \[
    |N^r_G[v] \cap S| \leq 2r+1.
    \]
\end{lemma}

\begin{proof}
    Let $S$ be a subgeodesic of $G$. 
    Let $G^+$ be a supergraph of $G$ and let $P$ be a geodesic with endpoints $s,t$ in $G^+$ such that $S \subseteq V(P)$.
    Let $v \in V(G)$.
    Suppose for contradiction that $2r+2\leq |N^r_G[v] \cap S|$.
    However, $N^r_{G}[v] \cap S \subseteq N^r_{G^+}[v] \cap S \subseteq N^r_{G^+}[v] \cap V(P)$.
    Let $x$ and $y$ be the vertices in $N^r_{G^+}[v] \cap V(P)$ closest to $s$ and $t$, respectively. 
    Since $N^r_{G^+}[v] \cap V(P) \subseteq xPy$, and 
    $|N^r_{G^+}[v] \cap V(P)|\geq 2r+2$, 
    we have $\dist_P(x,y) \geq 2r+1$.
    However, since $P$ is a geodesic in $G^+$, we have $2r+1\leq \dist_P(x,y) = \dist_{G^+}(x,y) \leq \dist_{G^+}(x,v) + \dist_{G^+}(v,y) \leq 2r$,
    a contradiction.
\end{proof}


\begin{proof}[Proof of \Cref{theorem:bound_wcol_Uhd_minor_free}]
    Let $h,d,r\geq 1$ and let $G$ be a $U_{h,d}$-minor-free graph.
    We will show that $\wcol(G) \leq 2\epsilon(h,d,0) \cdot (2r+1)\binom{\tau(h,0)+r}{\tau(h,0)} $, which implies the theorem.
    By \Cref{lemma:structure_for_wcol} applied to $G$ with $\ell=k=0$, 
    there is a graph $H$ with an ordering $\sigma_H=(x_1, \dots, x_{|V(H)|})$ of $V(H)$, and an $H$-partition $(V_x \mid x \in V(H))$ of $V(G)$ 
    such that \ref{item-p-first}-\ref{item-p-last} hold.
    Let $\sigma$ be a total order on $V(G)$ such that for all $i,j \in [|V(H)|]$ and $u,v \in V(G)$:
    \begin{enumerate}
        \item if $i < j$ and $u \in V_{x_i}, v \in V_{x_j}$, then $u<_\sigma v$;
        \item if $u \in A_{x_i}, v \in B_{x_i}$, then $u<_\sigma v$.
    \end{enumerate}

    Let $u \in V(G)$.
    Consider a vertex $v \in \WReach_r[G,\sigma,u]$.
    Let $i,j \in [|V(H)|]$ be such that $u \in V_{x_j}, v \in V_{x_i}$.
    Then $x_i \in \WReach_r[H,\sigma_H,x_j]$.
    In particular $i\leq j$.
    By \Cref{theorem:Grohe_et_al_wcol_bounded treewidth} 
    \[
    |\WReach_r[H,\sigma_H,x_j]| \leq \binom{r+\tau(h,0)}{\tau(h,0)}.
    \]

    Moreover $V_{x_i} = A_{x_i} \cup B_{x_i}$ where $|A_{x_i}| \leq \epsilon(h,d,0)$ and $B_{x_i}$ is the union of the vertex sets of at most $\epsilon(h,d,0)$ subgeodesics in $G[B_{x_i} \cup V_{x_{i+1}} \cup \dots \cup V_{x_{|V(H)|}}]$.
    Since vertices $r$-weakly reachable from $u$ in $B_{x_i}$ are in $N^r_{G[B_{x_i} \cup V_{x_{i+1}} \cup \dots \cup V_{x_{|V(H)|}}]}[u]$, we deduce by \Cref{lemma:intersection_ball_with_geodesic} that $|\WReach_r[G,\sigma,u] \cap B_{x_i}| \leq  \epsilon(h,d,0)\cdot(2r+1)$.
    Hence
    \begin{align*}
    |\WReach_r[G,\sigma,u] \cap V_{x_i}| &=
    |\WReach_r[G,\sigma,u] \cap A_{x_i}| + |\WReach_r[G,\sigma,u] \cap B_{x_i}| \\
    &\leq \epsilon(h,d,0) + (2r+1) \cdot \epsilon(h,d,0) \\
    &\leq (2r+1) \cdot 2\epsilon(h,d,0).
    \end{align*}
It follows that
    \[
    |\WReach_r[G,\sigma,u]| \leq \binom{r+\tau(h,0)}{\tau(h,0)} \cdot (2r+1) \ 2\epsilon(h,d,0).
    \]
    This proves the theorem.
\end{proof}

%% file: s.rest.tex
This section proves the following lemma.

\geodesics*

In short, \Cref{lem:X:geodesics} follows from a result by Pilipczuk and Siebertz~\cite{PS19}, see~\Cref{theorem:geodesic_partition_bounded_genus}, which we lift in order to accommodate vortical decompositions and clique-sums.

First, we recall the Graph Minor Structure Theorem of Robertson and Seymour~\cite{GMXVI}, which says that every graph in a proper minor-closed class can be constructed using four ingredients: graphs on surfaces, vortices, apex vertices, and tree-decompositions. 

The \emph{Euler genus} of a surface with~$h$ handles and~$c$ cross-caps is~${2h+c}$. The \emph{Euler genus} of a graph~$G$ is the minimum integer $g\geq 0$ such that there is an embedding of~$G$ in a surface of Euler genus~$g$; see \cite{MoharThom} for more about graph embeddings in surfaces. 

Let $G$ be a graph and let $\Omega$ be a cyclic permutation of a subset of $V(G)$. 
An \emph{interval of $\Omega$} is a sequence $(v_1, \dots, v_\ell)$ of vertices of $G$ such that $v_{i+1}$ is the successor of $v_i$ on $\Omega$ for every $i \in [\ell-1]$.
A \emph{vortical decomposition} of $G$ is a pair 
$(\Omega, (B_x\subseteq V(G) \mid x\in V(\Omega)))$ such that: 
\begin{enumerateNum}
\setcounter{enumi}{-1}
\item $x \in B_x$, for every $x\in V(\Omega)$, 
\item for each edge $uv\in E(G)$ there is $x\in\Omega$ with 
$u,v\in B_x$, and
\item for each vertex $v\in V(G)$ the set of vertices $x\in V(\Omega)$ with 
$v\in B_x$  induces a non-empty interval of $\Omega$.
\end{enumerateNum}
The \emph{width} of a vortical decomposition $(\Omega,(B_x\subseteq V(G) \mid x\in V(\Omega)))$ is defined to be $\max_{x\in V(\Omega)} |B_x|$.

For any integers $g,p,k,a\geq0$, a graph $G$ is \emph{$(g,p,k,a)$-almost-embeddable} if for some set $A\subseteq V(G)$ with $|A|\leq a$, there are graphs $G_0,G_1,\dots,G_s$ for some $0 \leq s \leq p$, cyclic permutations $\Omega_1,\ldots,\Omega_s$ of pairwise disjoint subsets of $V(G)$, and a surface $\Sigma$ of Euler genus at most $g$ such that:
\begin{enumerate}
\item $G-A = G_{0} \cup G_{1} \cup \cdots \cup G_s$;
\item $G_{1}, \dots, G_s$ are pairwise vertex-disjoint and non-empty;
\item for each $i\in[s]$, there is a vortical decomposition $(\Omega_i,(B^i_x\mid x\in V(\Omega_i)))$ of $(G_i,\Omega_i)$ of width at most $k$;
\item $G_{0}$ is embedded in $\Sigma$;
\item there are $s$ pairwise disjoint closed discs in $\Sigma$ whose interiors $\Delta_1,\ldots,\Delta_s$ are disjoint from the embedding of $G_0$, and such that the boundary of $\Delta_i$ meets the embedding of $G_0$ exactly in vertices of $\Omega_i$, and 
the cyclic ordering of $\Omega_i$ is compatible with the ordering of the vertices around the boundary of $D_i$, for each $i\in[s]$. 
\end{enumerate}
The vertices in $A$ are called \emph{apex} vertices. They can be adjacent to any vertex in $G$. 
For an integer $m\geq 0$, a graph is \emph{$m$-almost-embeddable} if it is $(m,m,m,m)$-almost-embeddable.

Let $G$ be a graph, let $\cgB=(T,(B_x\mid x\in V(T)))$ be a tree-decomposition of $G$. 
For $x \in V(T)$, the \emph{torso} of $B_x$, denoted by $\torso(G,\mathcal{B},x)$, is the graph obtained from $G[B_x]$ by adding edges so that $B_x\cap B_y$ is a clique for each neighbour $y$ of $x$ in $T$.

We now state the Graph Minor Structure Theorem, which is the cornerstone of structural graph theory. 

\begin{theorem}[\cite{GMXVI}]\label{thm:robertson-seymour-decomposition}
There exists a function $\alpha$ such that for every positive integer $t$, for every $K_t$-minor-free graph $G$, there exists a tree-decomposition $\mathcal{B}=(T,(B_x \mid x \in V(T)))$ of $G$ such that 
$\torso(G,\mathcal{B},x)$ is $\alpha(t)$-almost-embeddable, for every $x \in V(T)$.
\end{theorem}

The following result of Pilipczuk and Siebertz~\cite{PS19} is the starting point of our proof of~\Cref{lem:X:geodesics}.

\begin{theorem}[Theorem~18 in~\cite{PS19}]
\label{theorem:geodesic_partition_bounded_genus}
There exists a function $\zeta$ such that for every graph $G$ of Euler genus at most $g$, there is a partition $\mathcal{P}$ of $G$ into geodesics in $G$ such that $\tw(G/\mathcal{P}) < \zeta(g)$.
\end{theorem}

Pilipczuk and Siebertz~\cite{PS19} proved \Cref{theorem:geodesic_partition_bounded_genus} with $\zeta(g)= 16g +9$, which was later improved to  $\zeta(g)= 2g +7$ by Distel~et~al.~\cite{DHHW22}.

The next lemma lifts the previous statement to $(m,m,m,0)$-almost-embeddable graphs. This type of argument is folklore in the structural graph theory community.

We use the following convenient notation for manipulating paths in a graph. 
Let $G$ be a graph.
A \emph{walk} in $G$ is a sequence $(v_1, \dots, v_m)$ of vertices in $G$ such that $v_i v_{i+1} \in E(G)$ for each $i \in [m-1]$.
Let $U =(u_1, \dots, u_\ell)$ and $W=(w_1, \dots, w_m)$ be two walks in $G$ such that $u_\ell w_1 \in E(G)$ or $u_\ell=w_1$. The \emph{concatenation} of $U$ and $W$, denoted by $UW$, is the walk $(u_1, \dots, u_\ell,w_1, \dots, w_m)$ if $u_\ell w_1 \in E(G)$, or $(u_1, \dots, u_\ell,w_2, \dots, w_m)$ if $u_\ell = w_1$. Let $P$ be a path in $G$ and let $u,v$ be two vertices of $P$.
Define $uPv$ to be the subpath of $P$ from $u$ to $v$ (which is also a walk in $G$). This allows us to write expressions of the form $aPbcQdRe$ given that: 
$a,b,c,d,e$ are vertices in the graph; 
$P$ is a path containing $a$ and $b$; 
$bc$ is an edge;
$Q$ is a path containing $c$ and $d$;
$R$ is a path containing $d$ and $e$.

\begin{lemma}\label{lemma:geodesic_partition_k_almost_embeddable_graphs}
There is a function $\beta$ such that for every integer $m\geq 0$, for every $(m,m,m,0)$-almost-embeddable graph $G$, there is a partition $\mathcal{P}$ of $G$ into geodesics in $G$ such that $\tw(G/\mathcal{P}) < \beta(m)$.
\end{lemma}

\begin{proof}
    Let $\beta(m)=\zeta(m+2m-2)(11+3m)$ for all integers $m\geq0$, where $\zeta(\cdot)$ is the function given by~\Cref{theorem:geodesic_partition_bounded_genus}. 

    Fix $m\geq0$ and let $g$, $p$, $k$ be integers with $0\leq g,p,k\leq m$.
    Let $G$ be a $(g,p,k,0)$-almost-embeddable graph. 
    If $G$ is not connected, then we can process each component independently and take the union of the resulting partitions.
    Now assume that $G$ is connected.
    Let $s$, $G_0,G_1,\dots,G_s$, $\Omega_1, \dots, \Omega_s, \Sigma$, witness the fact that $G$ is $(g,p,k,0)$-almost-embeddable, and fix a vortical decomposition $(\Omega_i,(B^i_x\mid x\in V(\Omega_i)))$ of $G_i$ of width at most $k$, for every $i \in [s]$. 
    For convenience, we denote by $\Omega$ the permutation $\bigcup_{i \in [s]} \Omega_i$. 
    By definition, $\Omega_1, \dots, \Omega_s$ are pairwise disjoint, hence, for $x \in V(\Omega)$, we write $B_x  = B_x^i$ for the unique $i \in [s]$ such that $x \in V(\Omega_i)$.

    Let $G'$ be $G_0$ if $s=0$, and otherwise let $G'$ be a graph obtained from $G_0$ as follows:
    for every $i\in [s]$, for every pair $u,v$ of consecutive vertices on $\Omega_i$, if $uv \notin E(G_0)$, then add the edge $uv$
    (note that this is compatible with the embedding of $G_0$); 
    next pick arbitrarily a vertex $r\in V(\Omega)$ and 
    for all $v\in V(\Omega)-\set{r}$, if $rv \notin E(G_0)$, then add the edge $rv$. 
    Note that we may add $s-1$ handles to $\Sigma$, and embed $G'$ on the resulting surface, thus, $G'$ has Euler genus at most $g + 2(s-1) \leq g + 2p-2$.

    \begin{claim}
    $\dist_{G'}(u,v) \leq \dist_G(u,v) + 1$, for every $u,v \in V(G')$.
    \end{claim}
    \begin{proofclaim}
        Let $P$ be a geodesic in $G$ with endpoints $u$ and $v$.
        If $P$ intersects $V(\Omega)$ in at most one vertex, then $P$ is a path between $u$ and $v$ in $G'$, and so $\dist_{G'}(u,v) \leq \len(P) = \dist_G(u,v)$.
        Now suppose that $P$ contains at least two vertices in $V(\Omega)$.
        Let $u'$, $v'$ be such vertices that are closest in $P'$ to $u$ and $v$, respectively. 
        Then $uPu' r v'Pv$ is a walk from $u$ to $v$ in $G'$ of length at most $\len(P)+1 = \dist_G(u,v) + 1$, and so $\dist_{G'}(u,v) \leq \dist_G(u,v) + 1$.
    \end{proofclaim}

    \begin{claim}
        For every geodesic $P'$ in $G'$, $P'$ contains at most three vertices in $V(\Omega)$.
    \end{claim}
    \begin{proofclaim}
        Let $P'$ be a geodesic in $G'$ between $u$ and $v$.
        Suppose to the contrary that $P'$ has at least four vertices in $V(\Omega)$, and let $u'$, $v'$ be such vertices that are closest in $P'$ to $u$ and $v$, respectively. 
        Now $u'$ and $v'$ can be connected by a two-edge path via $r$ in $G'$. Therefore, $Q'=uP'u'rv'P'v$ is a walk in $G'$, and since there are at least two vertices on $P'$ between $u'$ and $v'$, the walk $Q'$ is shorter than $P'$, a contradiction. 
    \end{proofclaim}

    \begin{claim}
    For every geodesic $P'$ in $G'$, the vertex set of $P'$ is the union of the vertex sets of at most six disjoint geodesics in $G$, and moreover, each of these geodesics contains at most one vertex in $V(\Omega)$. 
    \end{claim}
    \begin{proofclaim}
        Let $P'$ be a geodesic in $G'$ between $u$ and $v$.
        Since $P'$ has at most three vertices in $\Omega$, it can be split into the disjoint union of at most three geodesics in $G'$ such that each part has at most one vertex in $\Omega$.

        Consider now a geodesic $Q$ in $G'$ with at most one vertex in $\Omega$. 
        The key property of $Q$ is that it is also a path $G$.
        We are going to prove that $Q$ can be split into at most two geodesics in $G$. 
        Let $a,b \in V(G_0)$ be the endpoints of $Q$.
        
        By a previous claim, 
        $\len(Q)=\dist_{G'}(a,b) \leq \dist_G(a,b) + 1$.
        Since $Q$ is a path in $G$ we also have $\dist_G(a,b) \leq \len(Q)$. Altogether,
        \[
        \len(Q) \in\set{\dist_{G}(a,b),\dist_{G}(a,b)+1}.
        \]
        If $\len(Q)=\dist_G(a,b)$ then $Q$ is a geodesic in $G$ and there is nothing to prove. 
        Now suppose that $\ell=\len(Q)=\dist_G(a,b)+1$.  
        Let $(q_0,\ldots,q_{\ell})$ be the walk along $Q$ with 
        $q_0=a$ and $q_{\ell}=b$. 
        For each $i\in\set{0,\ldots,\ell}$
        consider $d_i=\dist_G(q_0,q_i)$. 
        Note that 
        \[
         d_0=0,\  d_{\ell}=\ell-1,\ \textrm{and}\ d_{i}-d_{i-1}\in\set{-1,0,1}\ \textrm{for all $i\in[\ell]$} .
        \]
        These three conditions force that $d_{i}-d_{i-1}=1$ for all $i\in[\ell]$ except one value, say $j$, for which $d_{j}-d_{j-1}=0$. It follows, that $\dist_G(q_0,q_{j-1}) = j-1$, and $\dist_G(q_j,q_\ell) = \ell - j$, hence,
        \[
         aQq_{j-1}\ \textrm{and}\ q_{j}Qb\ \textrm{are geodesics in $G$.}
        \]
        This completes the proof that $Q$ can be split into at most two geodesics in $G$.

        Altogether, $P'$ is split into at most three times two geodesics in $G$, as desired.
    \end{proofclaim}

    Since $G'$ has Euler genus at most $g + 2(s-1) \leq g + 2p-2$, by \Cref{theorem:geodesic_partition_bounded_genus}, there is a partition $\mathcal{P}'$ of $G'$ into geodesics in $G'$ such that $\tw(G'/\mathcal{P}') < \zeta(g+2p-2)$.
    Let $(T,(W'_x \mid x \in V(T)))$ be a tree-decomposition of $G'/\mathcal{P}'$ of width at most $\zeta(g+2p-2)$.

    For each $P' \in \cgP'$, let $S(P')$ be a set of at most six geodesics in $G$ whose union of vertex sets is $V(P')$, and such that each of them intersects $V(\Omega)$ in at most one vertex. Define a partition of $V(G)$ into geodesics in $G$ by
    \[ 
        \cgP = \bigcup_{P' \in \cgP'} S(P') \cup \{\{u\} \mid u \in V(G) \setminus V(G_0)\} .
    \]
    We claim that $\tw(G/\mathcal{P}) < (\tw(G'/\mathcal{P}')+1)\cdot(6+3k) \leq  \zeta(g+2p-2)\cdot(6+3k)$.

    The family $\cgP$ is a partition of $G$ and the family $\cgP'$ is a partition of $G'$, thus, for each $u \in V(G)$ and $v \in V(G')$ we can define $P_u \in \cgP$ and $P_v' \in \cgP'$ to be such that $u \in P_u$ and $v \in P_v'$. 
    For each $x \in V(T)$, consider the following subsets of $\cgP$:
        \begin{align*}
            W_x^1 &= \bigcup_{P' \in W_x'} S(P'),\\
            W_x^2 &= \bigcup_{P' \in W_x'} \bigcup_{w \in V(\Omega) \cap V(P')} \{ P_v \mid v \in B_w \},\\
            W_x &=W_x^1 \cup W_x^2.
        \end{align*}

    Clearly, $|W_x^1| \leq 6|W_x'|$. Moreover, we proved that every geodesic in $G'$ has at most three vertices in $V(\Omega)$, thus, $|W_x^2| \leq 3k|W_x'|$. It follows that $|W_x| \leq |W_x'| \cdot (6+3k)$. Therefore, if we show that $(T,(W_x \mid x \in V(T)))$ is a tree-decomposition of $G/\mathcal{P}$, then indeed, $\tw(G/\mathcal{P}) < (\tw(G/\mathcal{P}')+1) \cdot (6+3k)$.

    \begin{claim}
        $(T,(W_x \mid x \in V(T)))$ is a tree-decomposition of $G/\mathcal{P}$.
    \end{claim}
    \begin{proofclaim}
    Let $u,v \in V(G)$ be such that $P_u$ and $P_v$ are distinct, and suppose that $uv \in E(G)$. If $uv \in E(G_0)$, then there exists $x \in V(T)$ such that $P_u',P_v' \in W'_x$. Moreover, $P_u \in S(P_u')$ and $P_v \in S(P_v')$. It follows that $P_u,P_v \in W_x^1 \subseteq W_x$. If $uv \in \bigcup_{i \in [s]}E(G_i)$ then there exists $w \in V(\Omega)$ such that $u,v \in B_w$. We have $P_w' \in W_x'$ for some $x \in V(T)$, and this yields $P_u,P_v \in W_x^2 \subset W_x$.

    It remains to show that for every $P \in \mathcal{P}$, the set $X_P = \{x \in V(T) \mid P \in W_x\}$ induces a non-empty, connected subset of $V(T)$. For every $P' \in \cgP$, 
    let $X'_{P'}$ be defined as $\{x \in V(T) \mid P' \in W_x'\}$. 
    Since $(T,(W'_x \mid x \in V(T))$ is a tree-decomposition of $G'/\mathcal{P}'$, we have that $X'_{P'}$ induces a non-empty, connected subset of $V(T)$. 
    Observe that the union $\bigcup_{w \in V(H')} X'_{P'_w}$, where $H'$ is a connected subgraph of $G'$, also induces a non-empty, connected subset of $V(T)$.
    For each $u \in V(G)$, let $I_u = \{w \in V(\Omega) \mid u \in B_w\}$.
    Since $V(G_1), \dots, V(G_s)$ are pairwise disjoint, and $(\Omega_i, (B_x \mid x \in V(\Omega_i)))$ is a vortical decomposition of $G_i$ for each $i \in [s]$,
    $I_u$ is either empty, or is an interval in some $\Omega_i$.
    Recall that we added cycle edges in $G'$ representing each $\Omega_i$, and hence, $I_u$ induces a connected subgraph in $G'$.
    
    First, suppose that $P = \{u\}$ for some $u \in V(G)\setminus V(G_0)$. By definition, 
        \[ X_{\{u\}} = \bigcup_{w\in I_u}X'_{P_w'}.\]
    Since $I_u$ is connected in $G'$, we conclude that 
    $X_{\{u\}}$ induces a non-empty, connected subset of $V(T)$.
    
    Now, suppose that $P \in S(P')$ for some $P' \in \cgP'$. Recall that $P$ contains at most one vertex in $V(\Omega)$. If $V(P) \cap V(\Omega) = \emptyset$, then $X_P = X'_{P'}$, which induces a non-empty, connected subtree of $T$. Otherwise, let $w$ be the unique vertex in the intersection $V(P) \cap V(\Omega)$. Then 
        \[
        X_P = X'_{P'} \cup \bigcup_{v \in I_w} X'_{P_v'}.
        \]
    Note that $w \in P'$, thus $P' = P_w'$, and since $w \in I_w$, we have $X_P = \bigcup_{v \in I_w} X'_{P_v'}$. This shows that $X_{P}$ induces a non-empty, connected subtree of $T$.
    \end{proofclaim}
    
    This completes the proof that $\tw(G/\mathcal{P})<\beta(m)=\zeta(m+2m-2)(11+3m)$.
\end{proof}

The next lemma is an immediate corollary of~\Cref{lemma:geodesic_partition_k_almost_embeddable_graphs} and \Cref{lemma:helly_property_tree_decomposition}. The function $\beta$ is the same as in \Cref{lemma:geodesic_partition_k_almost_embeddable_graphs}.

\begin{cor}\label{lemma:helly_property_tree_decomposition_almost_embeddable}
There exists a function $\beta$ such that for all integers $m\geq 0$ and $d \geq 1$, for every $(m,m,m,0)$-almost-embeddable graph $G$, for every family $\cgF$ of connected subgraphs of $G$ either:
    \begin{enumerate}
        \item there are $d$ pairwise vertex-disjoint subgraphs in $\mathcal{F}$, or
        \item there exist a subgraph $X$ of $G$ that is the union of at most $(d-1)\beta(m)$ geodesics in $G$, and for every $F \in \cgF$ we have $V(F) \cap V(X) \neq \emptyset$.
    \end{enumerate}    
\end{cor}

Consider a graph embedded in a fixed surface. 
It is clear that one can introduce parallel edges, subdivide edges of the graph, and the resulting graph still has an embedding into the surface. 
The point of the following observation is that we can do the same with $(g,p,k,a)$-embeddable graphs and the resulting graph has the same parameters except for the width of the vortices that may go up by $+2$.
This is folklore in the structural and algorithmic graph theory community.

\begin{obs}\label{lemma:subdivided_almost_embeddable_graph}
    Let $g,p,k,a$ be non-negative integers,
    and let $G$ be a $(g,p,k,a)$-almost-embeddable graph.
    For every graph $G'$ obtained from $G$ by duplicating some edges and then subdividing some edges, 
    $G'$ is $(g,p,k+2,a)$-almost-embeddable.
\end{obs}

The following observation says that almost-embeddability is preserved under taking subgraphs but, surprisingly, this may increase the width of the vortices. A proof can be found, for example, in \cite[Lemma~45]{DEMW23}.

\begin{obs}
Let $g,p,k,a$ be non-negative integers, and let $G$ be a $(g,p,k,a)$-almost-embeddable graph. 
Let $G'$ be a subgraph of $G$. 
Then $G'$ is $(g,p,2k,a)$-almost-embeddable.
\end{obs}

We have all the tools in hand to prove~\Cref{lem:X:geodesics}.

\begin{proof}[Proof of~\Cref{lem:X:geodesics}]
    Let $t,d$ be positive integers, let $G$ be a $K_t$-minor-free graph, and let $\cgF$ be a family of connected subgraphs of $G$. 
    If $\mathcal{F}$ is empty, then the result holds since $\beta(t),\gamma(t) \geq 0$. Now assume that $\mathcal{F}$ is non-empty.
    Without loss of generality, we can assume that each member of $\cgF$ is an induced subgraph. Therefore, with a slight abuse of notation, from now on we refer to $\cgF$ as a family of subsets of $V(G)$ such that each induces a connected graph.
    Suppose that \cref{lem:item:X:geodesics:disjoint} does not hold; that is, $\cgF$ has no $d$ pairwise disjoint members. In particular, $d\geq2$.
    
    Let $\alpha$ be the function from \Cref{thm:robertson-seymour-decomposition}. 
    By the theorem, there exists a tree-decomposition $\cgB = (T,(B_x \mid x \in V(T)))$ of $G$ such that $\torso(G,\cgB,x)$ is $\alpha(t)$-almost-embeddable, for every $x \in V(T)$. 
    For each $x \in V(T)$, let $A_x$ be the apex set of $\torso(G,\cgB,x)$ (that is, $A_x$ is a set of at most $\alpha(t)$ vertices such that $\torso(G,\cgB,x) - A_x$ is $(\alpha(t),\alpha(t),\alpha(t),0)$-almost-embeddable).
    By \Cref{lemma:helly_property_tree_decomposition}, there exists an integer $d' < d$ and $x_1, \dots, x_{d'} \in V(T)$ such that for every $F \in \mathcal{F}$, $F$ intersects $\bigcup_{i\in[d']} B_{x_i}$.
    Let $A = \bigcup_{i\in[d']} A_{x_i}$.
    Note that $|A| \leq (d-1)\alpha(t)$, so it suffices to take $\gamma = \alpha$. 
    For each $i \in [d']$, let $\cgF_i'$ be the family of all $F \in \cgF$ disjoint from $A$ that intersect $B_{x_i}$.
    

    We now sketch the next steps of the proof, see also \Cref{fig:gadget_graph}. First, for each $i \in [d']$ we modify the graph $G[B_{x_i}]$ to obtain an auxiliary graph $G_i^*$ that is $(\alpha(t),\alpha(t),2\alpha(t)+2,0)$-almost-embeddable. Then, we carefully project the family $\cgF'_i$ into $G_i^*$. In particular, when two sets from $\cgF'_i$ intersect, their projections will intersect as well. 
    Next, we will apply \Cref{lemma:helly_property_tree_decomposition_almost_embeddable} to the auxiliary graph to obtain a hitting set for the projected $\cgF'_i$ being a union of a small number of geodesics in $G_i^*$. Finally, we will lift the hitting set to the initial graph,  perhaps adding some more geodesics. Taking the union of hitting sets over all $i \in [d']$, we will finish the proof.

    Fix some $i \in [d']$, let $B = B_{x_i} - A$, and let $\cgF' = \cgF_i'$. 
    We say that two distinct vertices $u,v \in B$ are \emph{interesting} if $u$ and $v$ are in the same component of $G-A$ and there exists $y \in V(T)$ with $y \neq x_i$ such that $u,v \in B_y$. Let $\cgI$ be the set of all $2$-subsets of vertices in $B$ that are interesting. 
    
    We construct the auxiliary graph $G^*$ as follows. 
    We start the construction with $G[B]$. 
    For all $\{u,v\} \in \cgI$, 
    if $u$ and $v$ are adjacent in $G[B]$, then we call this length-one path $P_{uv}$ or $P_{vu}$; 
    if $u$ and $v$ are not adjacent in $G[B]$, then we add to the graph a path connecting $u$ and $v$ of length $\dist_{G-A}(u,v)$ where all internal vertices are new, i.e.\ disjoint from all the rest. Again, we call this path $P_{uv}$ or $P_{vu}$.
    Moreover, for all $\{u,v\} \in \cgI$, 
    we add to the graph a path connecting $u$ and $v$ of length $\dist_{G-A}(u,v)+1$ where all internal vertices are new. We call this path $P'_{uv}$ or $P'_{vu}$.
    This completes the construction of $G^*$.

    Note that $G^*$ is obtained from  $\torso(G,\cgB,x_i)$ by removing some vertices (from $A$), duplicating and perhaps subdividing some edges. 
    Therefore, by \Cref{lemma:subdivided_almost_embeddable_graph}, the graph $G^*$ is $(\alpha(t),\alpha(t),2\alpha(t)+2,0)$-almost-embeddable.

    Now, we will define a family $\cgF^*$ of connected subgraphs of $G^*$ that is roughly a projection of $\cgF'$ into $G^*$. 
    For a path $P$, let $\inter(P)$ denote the subpath of $P$ induced by all internal vertices of $P$. 
    For every $F \in \cgF'$, define
        \[
        F^* = (F \cap B) \cup \bigcup_{\substack{\{u,v\} \in \cgI\\ u,v \in F}} V(P_{u,v}) \cup \bigcup_{\substack{\{u,v\} \in \cgI\\ u \in F}} V(\inter(P_{u,v}')),
        \]
    and $\cgF^*=\set{F^*\mid F\in\cgF'}$. The following claim captures the critical properties of $\cgF^*$.

    \begin{figure}[!htbp] 
       \centering 
       \includegraphics[scale=1]{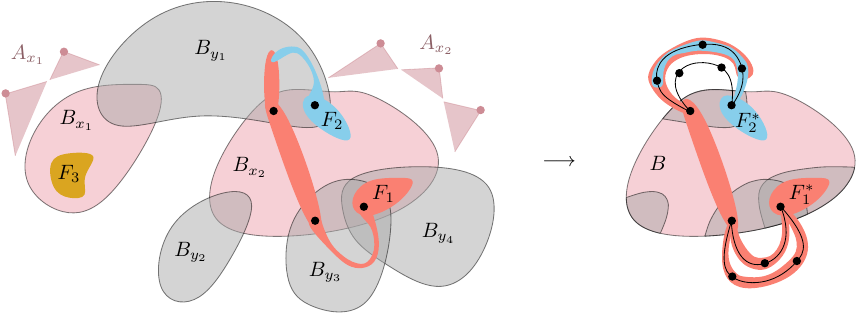} 
       \caption{The left figure depicts a tree-decomposition of a graph $G$.
       By \Cref{lemma:helly_property_tree_decomposition}, there is a small number of bags such that each member of $\cgF$ intersects these bags. 
       These are the bags $B_{x_1}$ and $B_{x_2}$.
       Next, we identify apex vertices (the set $A = A_{x_1} \cup A_{x_2}$).
       We focus on $B = B_{x_2}$ and define $\cgF'$ to be the elements of $\cgF$ that avoid $A$ and intersect $B$. 
       Note that we cannot just restrict the elements of $\cgF'$ to the graph $G[B]$, for two reasons. 
       First, $F \in \cgF'$ restricted to $G[B]$ can be disconnected.
       Second, $F_1,F_2 \in \cgF'$ that intersect in $G$ may no longer intersect when restricted to $G[B]$.
       We depict the two situations in the figure.
       To deal with these problems, we add some paths to the graph to obtain an auxiliary graph $G^*$ and we extend the subgraphs in $\cgF'$ to a family of subgraphs $\cgF^*$ of $G^*$.
       }
       \label{fig:gadget_graph}
    \end{figure}
    
    \begin{claim}
        Let $E,F \in \cgF'$. Then:
        \begin{enumerate}
            \item The graph $G^*[F^*]$ is connected. \label{item:projected_family_connected}
            \item If $E$ intersects $F$ then 
            $E^*$ intersects $F^*$.
            \label{item:projected_family_intersections}
        \end{enumerate}
    \end{claim}

    \begin{proofclaim}
        Let $u,v \in F^*$. 
        We will show that there is a path from $u$ to $v$ in $G^*[F^*]$ which will prove~\cref{item:projected_family_connected}. 
        If $u\not\in B$ then $u$ lies on one of the added paths in the construction of $G^*$. 
        Since each such path in $F^*$ has at least one endpoint in $B$, we can connect $u$ in $F^*$ with a vertex in $F^*\cap B$. 
        Therefore, we assume that both $u$ and $v$ are in $F^*\cap B$.

        Since $F\in\cgF'$, there is a walk $P$ connecting $u$  and $v$ in $G[F]$. 
        Recall that $F$ is disjoint from $A$, and so is $P$.
        We split $P$ into segments with endpoints in vertices from $B$, i.e., let $w_0,\ldots,w_{\ell}$ be vertices in $V(P)\cap B$ such that 
        \[
        P= w_0Pw_1\cdots Pw_{\ell-1}Pw_{\ell},
        \]
        where $w_0=u$, $w_{\ell}=v$ and $w_{j-1}Pw_{j}$ has no internal vertex in $B$ for each $j\in[\ell]$.
        Note that $w_{j-1}Pw_{j}$ could be just a one-edge path for some $j\in[\ell]$.

        We claim that we can replace each section $w_{j-1}Pw_j$ by a path connecting $w_{j-1}$ and $w_j$ in $G^*[F^*]$.
        Fix $j\in[\ell]$. 
  
        If $w_{j-1}$, $w_j$ are adjacent in $G[B]$, then they are also adjacent in $G^*$, as desired.
        If $w_{j-1}$ and $w_j$ are not adjacent in $G[B]$, 
        the set $X = \{y \in V(T) \mid B_y \cap V(\inter(w_{j-1}Pw_j)) \neq \emptyset\}$ induces a non-empty connected subset of $V(T)$.
        Moreover, since $w_{j-1}$ and $w_j$ are both adjacent to a vertex in $\inter(w_{j-1}Pw_j)$, there are vertices $y,y' \in X$ such that $w_{j-1} \in B_{y}$ and $w_j \in B_{y'}$.
        Since $X \cup \{x_i\}$ is acyclic in $T$, we have $y=y'$, and so
        $w_{j-1},w_j \in B_y$.
        This shows that $\{w_{j-1},w_j\} \in \mathcal{I}$.
        Thus, $P_{w_{j-1},w_j}$ was added to $F^*$ and we can use this path to connect $w_{j-1}$ and $w_j$ in $G^*[F^*]$.
        This way we completed a proof that there is a path from $u$ to $v$ in $G^*[F^*]$.

        Assume that $E, F \in \cgF'$ and that $E$ intersects $F$.
        To prove~\cref{item:projected_family_intersections}, 
        we will show that $E^*\cap F^*$ is non-empty as well.
        Fix $w\in E\cap F$. If $w\in B$, then $w \in E^*\cap F^*$, and we are done. 
        Hence, we suppose that $w\not\in B$.
        
        Let $P$ be a path in $G[E]$ from a vertex $u$ of $B$ to $w$ with no internal vertex in $B$.
        Let $Q$ be a path in $G[F]$ from $w$ to a vertex $v$ of $B$ with no internal vertex in $B$.
        If $u=v$, then $u \in E^* \cap F^*$ and we are done. 
        Otherwise we claim that $\{u,v\} \in \mathcal{I}$.
        Indeed, $\inter(PQ)$ is a non-empty connected subgraph of $G$, and so
        $X = \{x \in V(T)\mid V(\inter(PQ)) \cap B_x \neq \emptyset\}$
        is a non-empty connected subset of $V(T)$.
        Then, since $u$ and $v$ both have a neighbor in $\inter(PQ)$,
        we deduce that $u \in B_y, v \in B_{y'}$ for some $y,y' \in X \cap N_T(x_i)$.
        But since $T[X\cup\set{x_i}]$ is a tree, we must have $y=y'$, and so $u,v \in B_y$. This shows that $\{u,v\} \in \mathcal{I}$.
        Thus, $\inter(P'_{u,v}) \subseteq E^* \cap F^*$ and so $E^* \cap F^* \neq \emptyset$.
    \end{proofclaim}

    By the claim, the family $\cgF^*$ 
    is a family of connected subgraphs of $G^*$ containing no $d$ pairwise vertex-disjoint members. 
    Therefore, by \Cref{lemma:helly_property_tree_decomposition_almost_embeddable}, there exists a subgraph $X^*$ of $G^*$ such that $X^*$ is the union of a family $\cgR^*$ of at most $(d-1)\beta(2\alpha(t) + 2)$ geodesics in $G^*$ and for every $F \in \cgF'$ we have $V(F^*) \cap V(X^*) \neq \emptyset$.
    
    Let $R^* \in \cgR^*$. 
    Note that if one of the endpoints of $R^*$ lies on $\inter(P_{u,v})$ for some $\{u,v\} \in \cgI$, then one can remove $\inter(P_{u,v})$ from $R^*$ maintaining the fact that $\mathcal{R}^*$ is a family of geodesics in $G^*$ whose union of vertex sets intersects every member of $\mathcal{F}^*$. 
    Therefore, now assume without loss of generality that none of $R^* \in \cgR^*$ has an endpoint in the interior of any $P_{u,v}$. We now discuss the relation of geodesics in $G^*$ to the paths $P_{u,v}'$.

    \begin{claim}
        Let $\{u,v\} \in \cgI$.
        No geodesic in $G^*$ contains $P_{u,v}'$ as a subpath.
    \end{claim}
    \begin{proofclaim}
        Let $R^*$ be a geodesic in $G^*$. Suppose that it contains $P_{u,v}'$ as a subpath. Then replacing the segment corresponding to $P_{u,v}'$ in $R^*$ with $P_{u,v}$ gives a shorter walk between endpoints of $R^*$ in $G^*$, which is a contradiction.
    \end{proofclaim}

    We need the following easy observation. 

    \begin{claim}
        For all $u,v \in B$, we have
            $\dist_{G-A}(u,v) = \dist_{G^*}(u,v)$.
        Moreover, if $R^*$ is a geodesic in $G^*$ connecting $u$ and $v$, then there exists a geodesic $R$ in $G-A$ connecting $u$ and $v$ such that $V(R^*) \cap B \subset V(R) \cap B$.
    \end{claim}
    \begin{proofclaim}
        Let $u,v \in B$ and let $P$ be a path between $u$ and $v$ in $G-A$. We will show that there exists a path $P^*$ between $u$ and $v$ in $G^*$ of length at most the length of $P$. Let $w_0,\dots,w_\ell \in B$ and let $P_1,\dots,P_\ell$ be (possibly empty) paths in $G-A-B$ such that 
        \[
            P = w_0 P_1 w_1 P_2 \dots P_\ell w_\ell
        \]
        with $w_0 = u$ and $w_\ell=v$.
        Let $j \in [\ell]$. 
        If $P_j$ is an empty path, then let $P_j^*$ be also an empty path.
        Otherwise, $\set{w_{j-1},w_j}$.
        It follows that $P_{w_{j-1},w_j} \subset G^*$ and moreover, $\len(\mathrm{int}(P_{w_{j-1},w_j})) \leq  \len(P_j)$.
        Define $P_j^* = P_{w_{j-1},w_j}$. 
        Let $P^*$ be the walk defined by
        \[
            P^* = w_0 P_1^* w_1 \dots P_\ell^* w_\ell.
        \]
        Clearly, $P^*$ is a walk between $u$ and $v$ in $G^*$, and $\len(P^*) \leq \len(P)$.
        This shows that $\dist_{G^*}(u,v) \leq \dist_{G-A}(u,v)$.

        Now, let $P^*$ be a path between $u$ and $v$ in $G^*$. Let $w_0,\dots,w_\ell \in B$ and let $P_1^*,\dots,P_\ell^*$ be (possibly empty) paths in $G^*-B$ such that 
        \[
            P^* = w_0 P_1^* w_1 \dots P_\ell^* w_\ell
        \]
        with $u=w_0$ and $v=w_\ell$.
        If $P_j^*$ is an empty path, then let $P_j$ be also an empty path.
        Otherwise, by definition, it is clear that $\dist_{G-A}(w_{j-1},w_j) \leq \len(P_j^*)$. Let $P_j$ be any shortest path between $w_{j-1}$ and $w_j$ in $G-A$. 
        Let $P$ be the walk defined by
        \[
            P = w_0 P_0 w_1 \dots P_\ell w_\ell.
        \]
        Clearly, $P$ is a walk between $u$ and $v$ in $G-A$, and $\len(P) \leq \len(P^*)$. 
        This shows that $\dist_{G-A}(u,v) \leq \dist_{G^*}(u,v)$.
        
        Moreover, if $P^*$ is a geodesic in $G^*$, then
        $P$ is a geodesic in $G-A$ with $V(P^*) \cap B \subseteq V(P) \cap B$.
    \end{proofclaim}

    Let $\cgS$ be the collection of all the paths of the form $\mathrm{int}(P_{u,v}')$ in $G^*$ -- note that all such paths are nonempty.
    It follows that for every $R^* \in \cgR$, the geodesic $R^*$ intersects at most two distinct members of $\cgS$, and so, we can write that $R^*$ is a concatenation of $S_1$, $R_0^*$, and $S_2$, where $S_1$ and $S_2$ are subpaths of paths in $\cgS$ each, and $R_0^*$ is disjoint from $\bigcup_{S \in \cgS} V(S)$. 
    Clearly, $R_0^*$ is a geodesic in $G^*$, and moreover, it connects vertices in $B$. We aim to replace each geodesic $R^* \in \cgR$ with at most three geodesics in $G$ maintaining the property that the union of all constructed geodesics intersects every member of $\mathcal{F}'$. 

    For technical reasons, we assume that the empty path is a geodesic.

    \begin{claim}
        Let $R^* \in \cgR$. There exist at most three geodesics $F_{R^*}^0,F_{R^*}^1,F_{R^*}^2$ in $G$ such that for every $F \in \cgF'$, if $F^* \cap V(R^*) \neq \emptyset$ then $F \cap V(F_{R^*}^j) \neq \emptyset$ for some $j \in \{0,1,2\}$.
    \end{claim}
    \begin{proofclaim}
        Let $S_1,S_2,R_0^*$ be a partition of $R^*$ as described above.
        Let $u_1$ and $u_2$ be the endpoints of $R_0^*$. 
        By the previous claim, there exists a geodesic $R_0$ connecting $u_1$ and $u_2$ such that $V(R^*_0) \cap B \subset V(R_0) \cap B$. Put $F_{R^*}^0 = R_0$. Let $j \in \{1,2\}$. 
        If $S_j$ is an empty path, then set $F_{R^*}^j$ to be an empty path. 
        Otherwise, $S_j$ is a segment of the path $P_{u_j,u_j'}'$ for some $u_j' \in B$ such that $\{u_j, u_j'\} \in \mathcal{I}$. 
        In this case, set $F_{R^*}^j$ to be the one-vertex path containing $u_j'$.

        Clearly, $F_{R^*}^0,F_{R^*}^1,F_{R^*}^2$ are geodesics in $G-A$. 
        We now prove that they satisfy the assertion of the claim. 
        
        Let $F \in \cgF'$ be such that $F^* \cap V(R^*) \neq \emptyset$.
        Thus, either $F^* \cap V(S_j) \neq \emptyset$ for some $j \in \{1,2\}$, or $F^* \cap V(R_0^*) \neq \emptyset$.
        If $F^* \cap V(S_j) \neq \emptyset$ for some $j \in \{1,2\}$,
        then by the construction of $F^*$, either $u_j \in F$ or $u'_j \in F$.
        In the first case $F \cap V(F_{R^*}^0) \neq \emptyset$,
        and in the second case $F \cap V(F_{R^*}^j) \neq \emptyset$.

        It remains to deal with the case when $F^* \cap V(R_0^*) \neq \emptyset$. 
        By construction of $F^*$ we have $F^* \cap B \cap V(R_0^*) \neq \emptyset$.
        However $V(R^*_0) \cap B \subset V(R_0) \cap B$ and $F^* \cap B = F \cap B$. Therefore,
        \[ 
            F \cap V(F^0_{R^*}) \supseteq 
            F \cap V(R_0) \supseteq
            F \cap B \cap V(R_0) \supseteq 
            F^* \cap B \cap V(R_0) \supseteq 
            F^* \cap B \cap V(R_0^*) \neq
            \emptyset.
        \]
    \end{proofclaim}

    Finally, define
    \[
        X_i = \bigcup_{R^* \in \cgR} F_{R^*}^0\cup F_{R^*}^1\cup F_{R^*}^2.
    \]
    It follows that for each $i \in [d']$, the subgraph $X_i$ is the union of at most $3|\cgR|\leq 3(d-1)\beta(2\alpha(t) + 2)$ geodesics. Let $X = \bigcup_{i\in [d']} X_i$. For every $F \in \cgF$ we have $F \cap (X \cup A) \neq \emptyset$. Moreover, $X$ is a union of at most $3(d-1)^2\beta(2\alpha(t) + 2)$ geodesics in $G-A$. This proves the lemma with $\delta(t) = 3\beta(2\alpha(t) + 2)$.    
\end{proof}

%% file: s.exclude_an_apex.tex
Recall that a graph $G$ is \emph{apex} if there is a vertex $v\in V(G)$ such that $G-v$ is planar. For a given apex graph $X$, let $t(X)$ be the minimum integer such that, for some integer $c$, every $X$-minor-free graph is isomorphic to a subgraph of $H\boxtimes P\boxtimes K_c$ where $\tw(H)\leq t(X)$ and $P$ is a path. In this section, we show that $t(X)$ is tied to the treedepth of $X$.

A tree-decomposition $(T,(B_x\mid x\in V(T)))$ of a graph is \emph{rooted} when 
$T$ is a rooted tree.
For a rooted tree-decomposition $\mathcal{B} = (T,(B_x \mid x\in V(T)))$ of a graph $G$, let $\torso^-(G,\mathcal{B},x)$ be the supergraph of $G[B_x]$ obtained by adding all edges $uv$ with $u,v \in B_x \cap B_y$ and $x$ is the parent of $y$ in $T$.
We use the following result of Dujmović, Esperet, Morin, and Wood \cite{DEMW23}.


\begin{thm}[Theorem~48 in~\cite{DEMW23}]\label{apex_minor_free_structure}
For every apex graph $X$, there exist positive integers $w,t$ such that every $X$-minor-free graph $G$ has a rooted tree-decomposition $\cgB=(T,(B_x \mid x\in V(T)))$ of adhesion at most $3$, and for each $x\in V(T)$, there exists a layered partition $(\mathcal{P}_x,\mathcal{L}_x)$ of $\torso^-(G,\cgB,x)$ with:
\begin{enumerate}
    \item $|P\cap L|\le w$ for each $(P,L)\in\mathcal{P}_x\times\mathcal{L}_x$; \label{item:apex_minor_free_structure:layered_width}
    \item if $x$ has a parent $y$ in $T$, then
    \begin{enumerateAlpha}
        \item all vertices in $B_x\cap B_y$ are in the first layer of $\mathcal{L}_x$,\label{item:apex_minor_free_structure:interface_in_the_first_layer}
        \item each vertex of $B_x\cap B_y$ is in a singleton part of $\mathcal{P}_x$; and\label{item:apex_minor_free_structure:singleton_parts}
    \end{enumerateAlpha} \label{item:apex_minor_free_structure:parents_properties}
    \item $\torso^-(G,\cgB,x)/\mathcal{P}_x$ is a minor of $G$ and has treewidth less than $t$.\label{item:apex_minor_free_structure:Qx_minor_of_G}
\end{enumerate}
\end{thm}


The next result proves the upper bound in \eqref{ApexStructure}.

\begin{thm} 
\label{exluded-apex-graph}
    For every apex graph $X$, there exists a positive integer $c$ such that  for every $X$-minor-free graph $G$,
    there exists a graph $H$ of treewidth at most $2^{\td(X)+1}-1$ such that $G\subsetsim H\boxtimes P \boxtimes K_c$ for some path $P$. 
\end{thm}


\begin{proof}
    Let $X$ be an apex graph.
    Let $w,t$ be the constants depending only on $X$ given by \Cref{apex_minor_free_structure}.
    Let $c'$ be the constant depending only on $X$ given by \Cref{XMinorFreeProduct}.
    Let $c = c' \cdot t \cdot w$.

    Let $G$ be an $X$-minor-free graph.
    By \Cref{apex_minor_free_structure}, there is
    a rooted tree-decomposition $\cgB=(T,(B_x \mid x\in V(T)))$ of $G$ and for every $x \in V(T)$ there is a layered partition $(\mathcal{P}_x,\mathcal{L}_x)$ of $\torso^-(G,\mathcal{B},x)$ such that items~\ref{item:apex_minor_free_structure:layered_width}-\ref{item:apex_minor_free_structure:Qx_minor_of_G} hold.

    Let $r$ be the root of $T$. For each vertex $x$ in $T$ with $x\neq r$, let $p(x)$ be the parent of $x$ in $T$. 
    Let $(v_1, \dots, v_{|V(T)|})$ be an ordering of $V(T)$ such that for every edge $v_i v_j$ of $T$, if $v_i=p(v_j)$, then $i<j$.
    For every $i \in[|V(T)|]$, let $G^i$ be the graph obtained from $G[\bigcup_{j \leq i} B_{v_j}]$
    by adding for every $j >i$ with $p(v_j) \in \{v_1, \dots, v_i\}$, all the missing edges with both endpoints in $B_{v_j} \cap B_{p(v_j)}$.
    Next, for each $i\in[|V(T)|]$, we will construct a graph $H^i$, an $H^i$-partition $(V^i_x \mid x \in V(H^i))$ of $G^i$ and a layering $\mathcal{L}^i$ of $G^i$ such that
    \begin{enumerate}
        \item $\tw(H^i) \leq 2^{\td(X)+1}-1$, and
        \item $|V^i_x \cap L| \leq c$ for every $x \in V(H^i)$ and $L \in \mathcal{L}^i$.
    \end{enumerate}

    By \Cref{ObsLayeredPartitionProduct}, this yields $G \subsetsim H^{|V(T)|} \boxtimes P \boxtimes K_c$ for some path $P$, which will complete the proof.
    
    The construction is iterative, starting with $i=1$.    Observe that $v_1=r$ and $G^1 = \torso^-(G,\mathcal{B},r)$. 
    Let $Q=\torso^-(G,\mathcal{B},r)/\mathcal{P}_r$. 
    By \Cref{apex_minor_free_structure}.\ref{item:apex_minor_free_structure:Qx_minor_of_G}, $\tw(Q)<t$ and $Q$ is a minor of $G$, so $Q$ is $X$-minor-free.
    By \Cref{XMinorFreeProduct}, we obtain a graph $H^1$ and an $H^1$-partition $(U_z \mid z \in V(H^1))$ of $Q$
    such that $\tw(H^1) \leq 2^{\td(X)+1}-4$ and $|U_z| \leq c' \cdot t$ for every $z \in V(H^1)$.
    Let $V^1_z = \bigcup_{P \in U_z} P$ for every $z \in V(H^1)$ and $\mathcal{L}^1 = \mathcal{L}_r$. Then $(V^1_z \mid z \in V(H^1))$ is an $H^1$-partition of $G^1$ such that $|V^1_z \cap L| \leq |U_z| \cdot w \leq c' \cdot t \cdot w = c$ for every $z \in V(H^1)$ and $L \in \mathcal{L}^1$.

    Next, let $i>1$, and assume that $H^{i-1}$, $(V^{i-1}_x \mid x \in V(H^{i-1}))$ and $\mathcal{L}^{i-1}$ are already defined.
    Let $x = v_i$, $R = B_x \cap B_{p(x)}$, and  $Z = \{z \in V(H^{i-1}) \mid R \cap V^{i-1}_z \neq \emptyset\}$.
    Note that $R$ is a clique in $G^{i-1}$ and so $Z$ is a clique in $H^{i-1}$.
    Recall that the elements of $R$ are in singleton parts of $\mathcal{P}_x$ by \Cref{apex_minor_free_structure}.\ref{item:apex_minor_free_structure:parents_properties}.\ref{item:apex_minor_free_structure:singleton_parts}.
    Let $Q = \torso^-(G,\mathcal{B},x)/\mathcal{P}_x - \{\{v\} \mid v \in R\}$.
    By \Cref{apex_minor_free_structure}.\ref{item:apex_minor_free_structure:Qx_minor_of_G}, $\tw(Q)<t$ and $Q$ is a minor of $G$, so $Q$ is $X$-minor-free.
    By \Cref{XMinorFreeProduct}, we obtain a graph $H'$ and an $H'$-partition $(U_z \mid z \in V(H'))$ of $Q$
    such that $\tw(H') \leq 2^{\td(X)+1}-4$ and $|U_z| \leq c' \cdot t$ for every $z \in V(H')$.
    Now define $H^i$ to be the clique-sum of $H^{i-1}$ and $Z \oplus H'$
    according to the identity function on $Z$.
    Then $\tw(H^i) = \max\{\tw(H^{i-1}),|Z|+\tw(H')\} \leq 2^{\td(X)+1}-4+3$.
    For every $z \in V(H^i)$ let
    \[
    V^i_z = 
    \begin{cases}
        V^{i-1}_z & \text{if $z \in V(H^{i-1})$,} \\
        \bigcup_{P \in U_z} P & \text{if $z\in V(H')$.}\\    
    \end{cases}
    \]
    Then $(V^i_z \mid z \in V(H^i))$ is an $H^i$-partition of $G^i$.
    It remains to define the layering $\mathcal{L}^i = (L^i_0,L^i_1, \dots)$.
    Let $\mathcal{L}^{i-1} = (L^{i-1}_0,L^{i-1}_1, \dots)$
    and $\mathcal{L}_x = (L^x_0,L^x_1, \dots)$.
    Since $R$ is a clique in $G^{i-1}$, there is a non-negative integer $j$ such that $R \subseteq L^{i-1}_j \cup L^{i-1}_{j+1}$.
    For every non-negative integer $k$, let
    \[
    L^i_k = 
    \begin{cases}
        L^{i-1}_k &\text{if $k < j$,} \\
        L^{i-1}_k \cup (L^x_{k-j} - R) &\text{if $k\geq j$.} \\
    \end{cases}
    \]
    First we show that $\mathcal{L}^i=(L^i_0, L^i_1, \dots)$ is a layering of $G^i$.
    Let $uv$ be an edge of $G^i$. 
    Note that either $uv$ is an edge of $G^{i-1}$ or $uv$ is an edge of $\torso^-(G,\mathcal{B},x)$.  If $uv$ is an edge of $G^{i-1}$ then there is an integer $k$ such that $u,v \in L^{i-1}_k \cup L^{i-1}_{k+1}$, and so $u,v \in L^{i}_k \cup L^{i}_{k+1}$.
    If $uv$ is an edge of $\torso^-(G,\mathcal{B},x)$, then there is an integer $k$ such that $u,v \in L^x_k \cup L^x_{k+1}$.
    Note that in this case $\{u,v\} \not\subset R$.
    If $u,v \not\in R$, then $u,v \in (L^x_k-R) \cup (L^x_{k+1} -R) \subseteq L^i_{j+k} \cup L^i_{j+k+1}$.
    The last case to consider is when $|\{u,v\} \cap R| = 1$.
    Without loss of generality, assume that $u \in R$.
    By \Cref{apex_minor_free_structure}.\ref{item:apex_minor_free_structure:parents_properties}.\ref{item:apex_minor_free_structure:interface_in_the_first_layer}, $u \in R \subset L_0^x$, hence, $v \in L^x_0 \cup L^x_{1}$.
    Moreover, $u \in R \subset  L^{i-1}_{j} \cup L^{i-1}_{j+1}$.
    It follows that $u,v \in L^i_j \cup L^i_{j+1}$.
    This proves that $\mathcal{L}^i$ is a layering of $G^i$.
    
    Finally, for every non-negative integer $k$, and for every $z \in V(H^i)$, either $z \in V(H^{i-1})$ and 
    $|L^i_k \cap V^i_z| = |L^{i-1}_k \cap V^{i-1}_z| \leq c$, or $z \in V(H')$ and $|L^i_k \cap V^i_z| = |L^x_{k-j} \cap \bigcup_{P \in U_z} P| \leq |U_z| \cdot w \leq c' \cdot t \cdot w = c$.
    This concludes the proof.
\end{proof}

Dębski et al.~\cite{DFMS21} proved that if $G\subsetsim H \boxtimes P \boxtimes K_c$ where $\tw(H)\leq t$ and $P$ is a path, then $\chi_p(G) \leq c(p+1)\binom{p+t}{t} \leq c(p+1)^{t+1}$.

\Cref{exluded-apex-graph} thus implies:

\begin{cor}\label{cor:p_centered_coloring_apex}
For every apex graph $X$, every $X$-minor-free graph $G$, and every integer $p\geq 1$,
\[
\chi_p(G) \leq c(p+1)^{2^{\td(X)+1}},
\]
where $c$ is from \Cref{exluded-apex-graph}.
\end{cor}

%% file: s.open_questions.tex
We conclude the paper with a number of open problems. 

\begin{ques}
Can the upper bound on $\utw(\mathcal{G}_X) $ in \Cref{eq:bound_on_f} be improved? In particular, is $\utw(\mathcal{G}_X)$ at most a polynomial function of $\td(X)$?
\end{ques}

The next problem asks whether \Cref{XMinorFreeProduct} can be extended to the setting of excluded topological minors. 

\begin{ques}
\label{TopoQuestion}
    Is there a function $f$ such that
    for every graph $X$ there exists a function $c$ such that 
    for every positive integer $t$ and
    for every graph $G$ with $\tw(G)<t$ that does not contain $X$ as a topological minor,
    there exists a graph $H$ of treewidth at most $f(\td(X))$ such that 
    $G\subsetsim H\boxtimes K_{c(t)}$?
\end{ques}

This question is related to various results of Campbell~et~al.~\cite{CCDGHHHITTW22} on the underlying treewidth of $X$-topological minor-free graphs. They showed that a monotone class has bounded underlying treewidth if and only if it excludes some fixed topological minor. In particular, they proved the weakening of Question~\ref{TopoQuestion} with $\tw(H)\leq f(\td(X))$ replaced by $\tw(H)\leq |V(X)|$. This is tight for complete graphs. That is, the underlying treewidth of $K_t$-topological minor-free graphs equals $t$ (for $t\geq 5$), which implies Question~\ref{TopoQuestion} for complete graphs $X$. Campbell et al.~\cite{CCDGHHHITTW22} also prove \Cref{TopoQuestion}  for $X=K_{s,t}$ for $s\leq 3$, but note that it is open for $s\geq 4$.
They also prove that the underlying treewidth of $P_k$-free graphs equals $\floor{\log_2k}-1$, which gives good evidence for a positive answer to \Cref{TopoQuestion} since $\td(P_k) = \lceil \log_2(k+1) \rceil$.

A positive answer to \Cref{TopoQuestion} would be a qualitative generalisation of both \Cref{XMinorFreeProduct} and the following result of an anonymous referee of~\cite{Ding_1995} (where $X=K_{1,\Delta+1}$ in \Cref{TopoQuestion}): for every graph $G$ with treewidth $t$ and maximum degree $\Delta$,
there is a tree $T$ such that $G \subsetsim T \boxtimes K_{24 t \Delta}$. 

\begin{ques}
\label{ques:p-cetered}
Is there a function $g$ such that for every graph $X$, 
there is a constant $c$ such that for every $X$-minor-free graph $G$,
$\chi_p(G) \leq c\cdot p^{g(\td(X))}$ for every $p\geq 1$?\end{ques}

Our results give a positive answer to \Cref{ques:p-cetered} when $X$ is apex. However, we do not see a way to adjust our proof techniques and prove an analogue of~\Cref{thm-wcol-main} for $p$-centered colorings when $X$ is an arbitrary graph. The main obstacle is that we do not know how to use chordal partitions to construct $p$-centered colorings. Therefore, we do not know how to set up an equivalent of~\Cref{lemma:structure_for_wcol}.

\begin{ques}
Let $X$ be a graph. 
Let $f(X)$ be the infimum of all the real numbers $c$ such that there is a constant $a$, such that  for every $X$-minor-free graph $G$ and every integer $r \geq 1$, $\wcol_r(G) \leq a \cdot r^c$.
\Cref{thm-wcol-main} and a construction of Grohe~et~al.~\cite{Grohe15} imply that $\tw(X)-1 \leq f(X) \leq g(\td(X))$ for some function $g$. Is $f(X)$ tied to some natural graph parameter of $X$?
Is $f$ tied to some natural graph parameter?
\end{ques}

We know that $f$ is tied to neither $\td$, $\pw$ nor $\tw$. 
For treedepth or pathwidth, consider $X$ to be a complete ternary tree of vertex-height $k$ so both the pathwidth and treedepth of $X$ are $k$. 
Then $X$-minor-free graphs have bounded pathwidth, and it is easy to see that $\wcol_r(G)\leq (\pw(G)+1)(2r+1)$ for all graphs $G$. 
Thus, the exponent is $1$ which is independent of $k$.
For treewidth, consider the family $\set{G_{r,t}}_{r,t\geq 0}$ from~\cite{Grohe15}, which satisfy  $\tw(G_{r,t})\leq t$ and $\wcol_r(G_{r,t})=\Omega(r^t)$. 
Note that $G_{r,t}$ excludes $L_t$ (a ladder with $t$ rungs).
Since $\tw(L_t)\leq 3$ for all $t$, 
when we take $X=L_t$, the exponent becomes $t$ while treewidth remains constant.
The only parameter that we are aware of that could be tied with $f$ is $\td_2$, as defined in~\cite{HuynhJMSW22}.